\documentclass[12pt,reqno]{amsart}

\usepackage[margin=1in]{geometry}
\usepackage{graphicx}
\usepackage{amsmath, amssymb, amsthm} 
\usepackage{lmodern}
\usepackage[T1]{fontenc}
\usepackage{fancyhdr}
\usepackage[dvipsnames]{xcolor} 
\usepackage{pdfcolmk}
\usepackage{tikz-cd}
\usepackage{mathtools}
\usepackage{mathrsfs}  
\usepackage{subfiles}
\usepackage{calligra}
\usepackage{ytableau}
\usepackage{blkarray}

\usepackage[maxbibnames=99,backend=biber, style=alphabetic]{biblatex}
\usepackage{newclude}

\usepackage{caption}
\usepackage{subcaption}

\usepackage{style}

\usepackage{hyperref}
\hypersetup{
colorlinks = true,
citecolor = magenta,
linkcolor = black
}

\newcommand{\trip}{(\vec{b},\vec{e},\vec{a})}
\newcommand{\bal}{\operatorname{bal}}
\newcommand{\QuotSES}{\Quot^{n,d}_{\bP^1}(\cO(\vec{e}))}
\newcommand{\QuotLF}{\Quot^{n,d}_{\bP^1}(\cO(\vec{e}))^{\circ}}
\newcommand{\last}{\operatorname{last}}

\newcommand{\END}{\operatorname{end}}
\newcommand{\Pair}{\operatorname{Pair}}
\newcommand{\IB}{\operatorname{IB}}

\title{The locally free locus of Quot schemes on $\bP^1$}
\author{Feiyang Lin, Theodore Lysek}

\begin{document}

\begin{abstract}
    We characterize components of the locally free locus $\QuotLF$ of the Quot scheme associated to any vector bundle on $\bP^1$. Specifically, we show that the components are in bijection with certain combinatorial objects which we call strongly stable pairs. Using our explicit understanding of the components, we prove that $\QuotLF$ is connected, and we give an explicit bound for when $\QuotLF$ is irreducible. The key ingredient is a combinatorial criterion for when a triple of vector bundles on $\bP^1$ arises in a short exact sequence. As a consequence, we prove that in codimension $2$, all integral lattice points in the Boij-S\"oderberg cone are Betti diagrams of actual modules.
\end{abstract}

\maketitle

\section{Introduction}
Let $\vec{e} = (e_1, \dots, e_{m+n})$ where $m, n \geq 1$, $e_1 \leq e_2 \leq \cdots \leq e_{m+n}$, and define $\cO_{\bP^1}(\vec{e}) = \cO_{\bP^1}(e_1) \oplus \cdots \oplus \cO_{\bP^1}(e_{m+n})$. We often omit the subscript and write $\cO(\vec{e})$. Let $\QuotSES$ be the Quot scheme parametrizing flat families of quotients of $\cO(\vec{e})$ of rank $n$ and degree $d$. Let $\QuotLF \subseteq \QuotSES$ be the open locus where the universal quotient is locally free. 

When $\vec{e}$ is balanced (i.e. $e_{m+n}-e_1 \leq 1$), the Quot scheme $\Quot^{n,d}_{\bP^1}(\cO(\vec{e}))$ is a smooth, rational, projective variety \cite{Stromme}, but when $\cO(\vec{e})$ is not balanced, $\QuotSES$ can have many components. In \cite{Popa-Roth}, Popa and Roth showed that for any curve $C$ and any vector bundle $E$ on $C$, the Quot scheme $\Quot^{n,d}_{C}(E)$ is irreducible and generically smooth for $d$ sufficiently large. However, little is known about the geometry of the Quot scheme in general. In this article, we elucidate the behavior of $\QuotSES$ when $d$ is not sufficiently large by giving explicit characterizations of the components of $\QuotLF$, and showing that $\QuotLF$ is connected.

The key ingredient of our geometric results is an explicit understanding of which triples of vector bundles on $\bP^1$ can appear in a short exact sequence. Throughout, we work over a field $k$, with no restrictions on the characteristic.

\subsection{Characterization of short exact sequences} A tuple $\vec{f} = (f_1, \dots, f_m)$ of integers where $f_1 \leq f_2 \leq \cdots \leq f_m$ is called a \emph{splitting type}, corresponding to a rank $m$ vector bundle on $\bP^1$. We write $\rk \vec{f} = m$ and $\deg \vec{f} = \sum_{i = 1}^m f_i$. There is a partial order on splitting types which corresponds to the closure order in the universal setting of the stack of vector bundles on $\bP^1$. For splitting types $\vec{f}, \pvec{f}'$ of rank $m$ and the same degree, we write $\vec{f} \geq \pvec{f}'$, and say that $\vec{f}$ is more balanced than $\pvec{f}'$, if for all $1 \leq k \leq m$, $f_1 + \cdots f_k \geq f_1' + \cdots + f_k'$. 

A $k$-point of $\QuotLF$ corresponds to a locally free quotient [$\cO(\vec{e}) \twoheadrightarrow \cO(\vec{a})$], well-defined up to automorphisms of $\cO(\vec{a})$. Let the kernel be $\cO(\vec{b})$. Then the splitting types $(\vec{b}, \vec{a})$ of the kernel bundle and quotient bundle induce a stratification of $\QuotLF$.  Our strategy begins by studying which triples of splitting types $(\vec{b}, \vec{e}, \vec{a})$ are realizable in a short exact sequence
\[
0 \to \cO(\vec{b}) \to \cO(\vec{e}) \to \cO(\vec{a}) \to 0.
\]

\begin{definition}
    Let $\vec{a} = (a_1, \dots, a_n), \vec{b} = (b_1, \dots, b_m), \vec{e} = (e_1, \dots, e_{m+n})$ be splitting types of rank $m, n, $ and $m+n$ respectively. We say that the triple of splitting types $\trip$ is \emph{eligible} if 
    \begin{enumerate}
    \item $\rk \vec{a} + \rk \vec{b} = \rk \vec{e}$;
    \item $\deg \vec{a} + \deg \vec{b} = \deg \vec{e}$.
    \item there exists an injective map $\cO(\vec{b}) \to \cO(\vec{e})$ with locally free cokernel;
    \item there exists a surjective map $\cO(\vec{e}) \to \cO(\vec{a})$.
\end{enumerate}
\end{definition}
We recall the combinatorial versions of conditions (3) and (4) in Proposition \ref{prop:onlyif}. Clearly if $\trip$ is realizable in a short exact sequence, it must be eligible. To state the additional requirements for realizability, we first make some definitions. 
\begin{definition}
    Given $(\vec{b}, \vec{e}, \vec{a})$, for $1 \leq \mu \leq m, 1 \leq \nu \leq n$, let
\begin{align*}
    A(\mu, \nu) &=a_\nu - e_{\mu+\nu},\\
    B(\mu, \nu) &=e_{\mu+\nu-1} - b_\mu.
\end{align*}
For $1 \leq \mu \leq m+1, 1 \leq \nu \leq n+1$, let
\[
    S(\mu, \nu) = \sum_{i=\nu}^n a_i + \sum_{i=\mu}^m b_i - \sum_{i=\mu+\nu-1}^{m+n}e_i.
\]
\end{definition}

\begin{theorem} \label{thm:realizability_intro}
    An eligible triple $\trip$ is realizable in a short exact sequence if and only if the following conditions are satisfied:
    \begin{enumerate}
     \item Given $1 \leq \mu \leq m$, if there exists $1 \leq \nu' \leq n$ such that $A(\mu, \nu') < 0$, let \[\nu = \max \{1 \leq \nu' \leq n: A(\mu, \nu') < 0\}.\] Then $S(\mu, \nu+1) \geq 0$.
    \item Given $1 \leq \nu \leq n$, if there exists $1 \leq \mu' \leq m$ such that $B(\mu', \nu) < 0$, let \[\mu = \min \{1 \leq \mu' \leq m: B(\mu', \nu) < 0\}.\] Then $S(\mu, \nu+1) \geq 0$.
    \end{enumerate}
\end{theorem}
In the body of the paper, we also provide several other equivalent characterizations of realizability (Theorem \ref{thm:main_realizability}, Theorem \ref{thm:realzable_iff_balancing}). Our results recover that of Mandal \cite{Mandal2021}, who studied the case where $\vec{e}$ is perfectly balanced. In that case, a closed point of $\Quot^{n,d}_{\bP^1}(\cO^{\oplus (m+n)})^{\circ}$ corresponds to a map from $\bP^1$ to $\Gr(m,m+n)$ of degree $d$. In general, a closed point of $\QuotLF$ corresponds to a section of the Grassmann bundle $\Gr(m, \cO(\vec{e})) \to \bP^1$ of degree $d$.

\subsection{Connections to Boij-S\"oderberg theory} In Boij-S\"oderberg theory, one is interested in the Betti diagrams of free resolutions of graded Cohen-Macaulay modules of a fixed codimension $c$ over a polynomial ring. The perspective there says that if we only care about Betti diagrams up to scaling, then we should really consider the cone of rational multiples of Betti diagrams, viewed as vectors in an appropriate vector space over $\Q$. Boij and S\"oderberg conjectured explicit descriptions of this cone in terms of the extremal rays \cite{Boij-Soderberg2008}, which was proven by Eisenbud and Schreyer \cite{eisenbud2009betti}. 

While all Betti diagrams of free resolutions of graded modules correspond to an integral lattice point in the Boij-S\"oderberg cone, the converse is not necessarily true. The smallest example known to us is in codimension $3$, where the betti diagram 
\[
D = \begin{pmatrix}
    1 & 2 & - & - \\ 
    - & - & 2 & 1 
\end{pmatrix}
\]
lies in the Boij-S\"oderberg cone, but is not the Betti table of an actual module because the two first syzygies would satisfy a linear Koszul relation. However,
\[
2D = \begin{pmatrix}
    2 & 4 & - & - \\ 
    - & - & 4 & 2 
\end{pmatrix}
\]
is the Betti table of an actual graded module. One can generate such a module and check its resolution has the correct betti table using the following code in Macaulay2:
\begin{verbatim}
    loadPackage "BoijSoederberg"
    L = {0,1,3,4}
    M = randomModule(L, 2)
    betti res M
\end{verbatim}

It is thus interesting to characterize those integral lattice points in the Boij-S\"oderberg cone which are Betti diagrams of genuine graded modules. Our characterization of short exact sequences on $\bP^1$ can equivalently be viewed as characterizing Betti diagrams of free resolutions of Cohen-Macaulay graded modules over $k[x,y]$ of codimension $2$. This allow us to prove a conjecture of Daniel Erman, which says that examples like the above do not occur in codimension $2$.
\begin{corollary}[\cite{Erman2010}, Conjecture 2.5.4]
In projective dimension 2, every integral lattice point in the cone of Betti diagrams is the Betti diagram of a genuine module.
\end{corollary}
\begin{proof}
    Let $\beta$ be a virtual Betti diagram, i.e. an integral lattice point in the cone of Betti diagrams. Then $\beta$ can uniquely be written as a positive rational combination of pure Betti diagrams. In particular, by taking the positive integer multiple $M$ clearing the denominator in all of the rational coefficients, we see that $M\beta$ is the Betti diagram of a genuine module, corresponding to the direct sum of the pure Betti diagrams with the multiplicities as prescribed by the coefficients. However, the realizability criterion above tells us that if $M \beta$ is realizable, then so is $\beta$.
\end{proof}

\subsection{Geometric properties of \texorpdfstring{$\QuotLF$}{the locally free locus}}
Our first geometric result is a characterization of the components of $\QuotLF$. The key observation is that for a given component of $\QuotLF$, there is a generic behavior for the splitting types of the kernel and quotient bundles, which must be distinguished among all possible realizable triples appearing in $\QuotLF$.

\begin{definition}[Combinatorially stable pair] \label{def:combo_stable_intro}
    A combinatorially stable pair is associated to the following package of data:
    \begin{enumerate}
        \item Nonnegative integers $m', n'$, such that $0 \leq m' \leq m, 0 \leq n' \leq n$;
        \item A partition $\Lambda=(P_1,Q_1,\dots,P_r,Q_r)$ of $(n'+1,\dots,m+n-m')$ into $2r$ parts such that $e_{P_{i,\text{last}}}<e_{Q_{i+1,1}}$ for each $1\leq i\leq r-1$. Given such data, we define 
        \begin{align*}
            \Delta P_i & = \sum_{1\leq j\leq |P_i|}(e_{P_{i,j}}-e_{P_{i,1}-1}-1)\\ 
            \Delta Q_i &= \sum_{1\leq j\leq |Q_i|}(e_{Q_{i,\last}+1}-e_{Q_{i,j}}-1). 
        \end{align*}
        
        \item A tuple $\vec{\delta}=(\delta_1,\dots,\delta_r)$, where $0 \leq \delta_i \leq \min(\Delta Q_i, \Delta P_i)$.
    \end{enumerate}
Given such data, let $\alpha(Q_i,\delta_i)$ be the most balanced tuple that is at least $e_{Q_i}$ in each entry and has $\sum_j(\alpha(Q_i,\delta_i)_j-e_{Q_{i,j}})=\delta_i$. Similarly, let $\beta(P_i,\delta_i)$ be the most balanced tuple that is at most $e_{P_i}$ in each entry and has $\sum_j(e_{P_{i,j}} -\beta(P_i,\delta_i)_j)=\delta_i$. Then the \emph{combinatorially stable pair} associated to the package above is obtained by concatenating these balanced tuples as follows:
\begin{align*}
\vec{a} &= (e_1, \dots, e_{n'}, \alpha(Q_1, \delta_1), \dots, \alpha(Q_r, \delta_r)) \\ 
\vec{b} &= (\beta(P_1, \delta_1), \dots, \beta(P_r, \delta_r), e_{m+n-m'+1}, \dots, e_{m+n})
\end{align*}
\end{definition}

\begin{definition}[Strongly stable pairs]
    A \emph{strongly stable pair} is a combinatorially stable pair satisfying the additional conditions that 
    \begin{align*}
    e_{n'} & < \beta(P_1, \delta_1)_1, \\ 
    \alpha(Q_i, \delta_i)_{\last} & < \beta(P_{i+1}, \delta_{i+1})_1 \text{ for } 1 \leq i < r,\\ 
    \alpha(Q_r, \delta_r)_{\last}& < e_{m+n-m'+1}. 
    \end{align*} 
\end{definition}
Note that in the body of this article, this characterization will be stated as a theorem instead of the definition for strongly stable pairs.

\begin{theorem}[Theorem \ref{thm:strongly_stable}] \label{thm:intro_components}
By taking the generic splitting types for the kernel and quotient bundles, components of $\QuotLF$ correspond exactly to the strongly stable pairs.
\end{theorem}

The list of strongly stable pairs for a fixed Quot scheme can be obtained efficiently. Below, we list all the strongly stable pairs for an explicit example.
\begin{example}[Example \ref{ex:worked_out_stable_pairs}]
    In the case of $\vec{e} = (0,4,5,6,8,12), d = 20, n = 3$, there are five strongly stable pairs, corresponding to five components of $\Quot^{3,20}(\cO(\vec{e}))^{\circ}$. We also show that every component of $\QuotLF$ is generically smooth (Proposition \ref{prop:strataQuotLF}). As such, the dimension of a component $Z_{\vec{b}, \vec{a}}$ corresponding to the pair $(\vec{b}, \vec{a})$ should equal the dimension of the tangent space of a generic point of the component, which is $ \hom(\cO(\vec{b}), \cO(\vec{a}))$. We use this to compute the dimension of $Z_{\vec{b}, \vec{a}}$ below.
    \begin{align*}
        & \vec{b}=(4,5,6), \ \vec{a}=(0,8,12), \dim Z_{\vec{b}, \vec{a}} = 36\\ 
        & \vec{b}=(5,5,5),  \ \vec{a}=(0,4,16),  \dim Z_{\vec{b}, \vec{a}} = 36 \\ 
        & \vec{b}=(0,3,12),  \ \vec{a}=(6,6,8), \dim Z_{\vec{b}, \vec{a}} = 37\\ 
        & \vec{b}=(1,2,12),  \ \vec{a}=(0,10,10),  \dim Z_{\vec{b}, \vec{a}} = 38\\ 
        & \vec{b}=(-5,8,12),  \ \vec{a}=(6,7,7), \dim Z_{\vec{b}, \vec{a}} = 38
    \end{align*}
    As can be seen from this example, $\QuotLF$ is in general not equidimensional.
\end{example}

As a corollary, we obtain an explicit range for the degrees at which $\QuotLF$ is irreducible. Let $d' = \deg \cO(\vec{e}) - d$, which is the degree of the kernel bundle.
\begin{corollary} \label{cor:intro_irreducibility}
    $\QuotLF$ is irreducible if $d\geq n(e_{m+n}-1)+1$ and $d' \leq m(e_1+1)-1$.
\end{corollary}
It would be interesting to check if $\QuotSES$ itself is irreducible when the above bounds are satisfied. Note that in general there are other intermediate degrees where $\QuotLF$ is irreducible. However, we expect those degrees to be coincidental, resulting from the omission of components where the quotient sheaf is not locally free generically.

Using our understanding of the components of $\QuotLF$, we prove the following:
\begin{theorem}\label{thm:intro_connected}
    For all $\vec{e}, n, d$, $\QuotLF$ is connected.
\end{theorem}

\noindent
{\textbf{Acknowledgements.}} 
    Eric Larson suggested that realizability may be equivalent to the existence of a balancing of $\vec{b}$ against $\vec{a}$ to obtain $\vec{e}$, which turned into Theorem \ref{thm:realzable_iff_balancing}. Daniel Erman pointed out the connection of our results to Boij-S\"oderberg theory. We thank David Eisenbud, Frank-Olaf Schreyer and Hannah Larson for valuable discussions. Part of this work was conducted while the first author was supported by NSF Grant DMS 2001649. 

\section{Realizable Short Exact Sequences on \texorpdfstring{$\bP^1$}{the projective line}}
In this section, we give necessary and sufficient conditions for a triple of splitting types $(\vec{b}, \vec{e}, \vec{a})$ to be \textit{realizable}, i.e. for them to arise in a short exact sequence of vector bundles on $\bP^1$ like so:
\begin{equation} \label{eq:ses}
    0 \to \cO(\vec{b}) \to \cO(\vec{e}) \to \cO(\vec{a}) \to 0.
\end{equation}

\subsection{Necessary Conditions}
Let $\vec{b}=(b_1,\dots,b_m), \vec{e}=(e_1,\dots,e_{m+n}), \vec{a}=(a_1,\dots,a_n)$. Throughout this paper, we think of the quotient map $\cO(\vec{e})\to\cO(\vec{a})$ as an $n \times (m+n)$ matrix whose $(i,j)$-th entry is a homogeneous polynomial of degree $a_i-e_j$ in $k[x,y]$, and similarly for the kernel map $\cO(\vec{b}) \to \cO(\vec{e})$. In order for $\trip$ to be realizable, we need:
\begin{enumerate}
    \item $\rk \vec{a} + \rk \vec{b} = \rk \vec{e}$;
    \item $\deg \vec{a} + \deg \vec{b} = \deg \vec{e}$;
    \item there exists an injective map $\cO(\vec{b}) \to \cO(\vec{e})$ with locally free cokernel;
    \item there exists a surjective map $\cO(\vec{e}) \to \cO(\vec{a})$.
\end{enumerate}
As a convenient way to refer to this package of conditions, we say that a triple $\trip$ is \textit{eligible} if it satisfies all four conditions above. The latter two conditions are also numerical in terms of the entries of $\trip$, due to the following proposition:
\begin{proposition}[Theorem A.1.1, Hong-Larson] \label{prop:onlyif}
Let $\vec{e} = (e_1, \dots, e_r)$, $\vec{a} = (a_1, \dots, a_k)$ where $k < r$, $e_i \leq e_{i+1}$, $a_i \leq a_{i+1}$. Then there exists a surjection $\cO(\vec{e}) \twoheadrightarrow \cO(\vec{a})$ if and only if for all $1 \leq i \leq k$, either $a_i \geq e_{i+1}$, or for all $1 \leq j \leq i$, $a_j = e_j$.

Dually, let $\vec{b} = (b_1, \dots, b_k)$. Then there exists an injective map of vector bundles $\cO(\vec{b}) \hookrightarrow \cO(\vec{e})$ (i.e. the cokernel is a vector bundle) if and only if for all $1 \leq i \leq k$, either $b_i \leq e_{r-k+i-1}$, or for all $i \leq j \leq k$, $b_j = e_{r-k+j}$. \qed
\end{proposition}

\begin{figure}
\begin{minipage}{.4\textwidth}
  \centering
  \[
        \begin{bmatrix}
            1 & - & - & - & - & - & -\\ 
            * & 1 & - & - & - & - & -\\ 
            * & * & * & {\color{red}*} & - & - & -\\ 
            * & * & * & * & {\color{red}*} & - & -\\ 
            * & * & * & * & * & {\color{red}*} & -\\ 
        \end{bmatrix}
    \]
\end{minipage}
\begin{minipage}{.4\textwidth}
  \centering
   \[
        \begin{bmatrix}
            - & - & - & -  \\ 
            - & - & - & - \\ 
            {\color{red}*} & - & - & -  \\ 
            * & {\color{red}*} & - & -  \\ 
            * & * & - & -  \\ 
            * & * & 1 & -  \\ 
            * & * & * & 1  
        \end{bmatrix}
    \]
\end{minipage}
    \centering
    \caption{The shape of a matrix $\cO(e_1, \dots, e_7) \to \cO(a_1, \dots, a_5)$, which satisfies the proposition above via $a_1 = e_1, a_2 = e_2$, and $a_3 \geq e_4, a_4 \geq e_5, a_5 \geq e_6$. Similarly, we have a matrix $\cO(b_1,\dots, b_4) \to \cO(e_1, \dots, e_7)$ that satisfies the proposition above via $b_4 = e_7, b_3 = e_6, b_2 \leq e_4, b_1 \leq e_3$. An entry $1$ means that the labeled entry is a unit (degree is zero), and a star $*$ means that the entry can be filled with some nonzero polynomial (degree is nonnegative).} 
    \label{fig:matrixShape}
\end{figure}

We give two examples of the shapes of matrices satisfying the conditions above in Figure \ref{fig:matrixShape}. Visually, the conditions are saying that except for when we have identical entries, one is required to be able to fill entries lying on some minimal superdiagonal. 

While eligibility is necessary for realizability, it will often be mostly subsumed by other conditions required for realizability, so the following notion will be convenient:
\begin{definition}
    We say a triple $(\vec{b},\vec{e},\vec{a})$ is \textit{weakly eligible} if:
    \begin{enumerate}
        \item $\rk \vec{a} + \rk \vec{b} = \rk \vec{e}$;
        \item $\deg \vec{a} + \deg \vec{b} = \deg \vec{e}$;
        \item $a_i\geq e_i$ for all $1 \leq i \leq n$;
        \item $b_i\leq e_{n+i}$ for all $1\leq i\leq m$.
    \end{enumerate}
\end{definition}

In general, however, there are additional necessary conditions. Many forced zeros in the quotient matrix impose further constraints on the splitting type of the kernel, and vice versa when the kernel matrix has many zeros. To state this more precisely, let us first define a few quantities.
\begin{definition}
    Given $(\vec{b}, \vec{e}, \vec{a})$, for $1 \leq \mu \leq m, 1 \leq \nu \leq n$, let
\begin{align*}
    A(\mu, \nu) &=a_\nu - e_{\mu+\nu},\\
    B(\mu, \nu) &=e_{\mu+\nu-1} - b_\mu,\\
    S(\mu, \nu) &= \sum_{i=\nu}^n a_i + \sum_{i=\mu}^m b_i - \sum_{i=\mu+\nu-1}^{m+n}e_i, \\
    T(\mu, \nu) &= \sum_{i=1}^\nu a_i + \sum_{i=1}^\mu b_i - \sum_{i=1}^{\mu+\nu}e_i.
\end{align*}
We allow $\mu$ and $\nu$ to be $m+1$ and $n+1$ for $S(\mu,\nu)$ and we also allow either to be zero for $T(\mu,\nu)$. Note that for $0 \leq \mu \leq m, 0 \leq \nu \leq n$ we have $S(\mu+1,\nu+1)+T(\mu,\nu)=0$. 
\end{definition}

\begin{lemma}
\label{lem:necc} 
Suppose $\trip$ is realizable.
\begin{enumerate}
     \item Given $1 \leq \mu \leq m$, if there exists $1 \leq \nu' \leq n$ such that $A(\mu, \nu') < 0$, let 
     \[\nu = \max \{1 \leq \nu' \leq n: A(\mu, \nu') < 0\}.\] Then $S(\mu, \nu+1) \geq 0$; equivalently, $T(\mu-1, \nu) \leq 0$.
    \item Given $1 \leq \nu \leq n$, if there exists $1 \leq \mu' \leq m$ such that $B(\mu', \nu) < 0$, let 
    \[\mu = \min \{1 \leq \mu' \leq m: B(\mu', \nu) < 0\}.\] Then $S(\mu, \nu+1) \geq 0$; equivalently, $T(\mu-1, \nu) \leq 0$.
\end{enumerate}
\end{lemma}

\begin{proof}
(1) 
The inequality $S(\mu, \nu+1) \geq 0$ is a closed condition on the space of surjections from $\cO(\vec{e})$ to $\cO(\vec{a})$, which is irreducible. So it suffices to prove the lemma for a general map $f: \cO(\vec{e}) \to \cO(\vec{a})$. 
When $A(\mu, \nu) = a_\nu - e_{\mu+\nu}< 0$, such a map $f$ must have many forced zeros, as illustrated below:
\[
    \begin{blockarray}{ccccccc}
        & e_1 & \dots & e_{\mu+\nu-1} & e_{\mu+\nu} & \dots & e_{m+n} \\
      \begin{block}{c(cccccc)}
        a_1 & * & \dots & * & 0 & \dots & 0\\
        \vdots &  & \dots &  &  &  & \\
        a_{\nu} & * & \dots & * & 0 & \dots & 0\\
        a_{\nu+1} & * & \dots & * & * & \dots & *\\
        \vdots &  & \dots &  &  &  & \\
        a_{n } & * & \dots & * & * & \dots & *\\
      \end{block}
    \end{blockarray}.
\]
Let $f':\cO(e_{\mu+\nu}, \dots, e_{m+n}) \to \cO(a_{\nu+1}, \dots, a_{n})$ be the map given by the lower right block. Since $f$ is general, we may assume that the entries of the lower right block are general. By assumption, for all $\nu' \geq \nu+1,$ we have $a_{\nu'}\geq e_{\mu+\nu'}$, so $f'$ is surjective, as no entries on its diagonal or superdiagonal are forced to be zero. Then the kernel of $f'$ has rank $m-\mu+1$ and degree 
\[\sum_{i=\mu+\nu}^{m+n} e_i - \sum_{i=\nu+1}^{n} a_i.\]
The kernel of $f$ must contain the kernel of $f'$ as a subsheaf, which means that the last $m-\mu+1$ entries of $\vec{b}$ must sum to at least the degree of the kernel of $f'$. So \[\sum_{i=\mu}^{m} b_i \geq \sum_{i=\mu+\nu}^{m+n} e_i - \sum_{i=\nu+1}^{n} a_i.\] 
Equivalently,
\[S(\mu,\nu+1) = \sum_{i=\nu+1}^{n} a_i + \sum_{i=\mu}^{m} b_i - \sum_{i=\mu+\nu}^{m+n} e_i \geq 0.\]

(2) Dualizing the short exact sequence and applying part (1) gives  the desired claim. 
\end{proof}

\begin{lemma}
    If a triple $\trip$ satisfies the conditions of Lemma \ref{lem:necc}, and is weakly eligible, then it is eligible. 
\end{lemma}
\begin{proof}
    Let us consider part (1) of Lemma \ref{lem:necc} when $\mu = 1$. If there exists $\nu'$ such that $A(1, \nu') = a_{\nu'} - e_{\nu' + 1} < 0$, then the lemma says that if we let $\nu = \max\{1 \leq \nu' \leq n: a_{\nu'} - e_{\nu' + 1} < 0\}$, then $S(1, \nu+1) \geq 0$, or equivalently $T(0, \nu) \leq 0$, which says that
\[
(a_1 + \cdots + a_{\nu}) - (e_1 + \cdots + e_{\nu}) \leq 0.
\]
Weak eligibility says we must have that $a_1 \geq e_1, \dots, a_\nu \geq e_{\nu}$. In particular, this recovers the Hong-Larson criterion (Proposition \ref{prop:onlyif}) that there exists $\nu$ such that $a_1 = e_1, \dots, a_{\nu} = e_{\nu}$, and $a_{i} \geq e_{i+1}$ for $i > \nu$. Similarly, if we consider part (2) of Lemma \ref{lem:necc} where we set $\nu = n$, one recovers the dual statement of Hong-Larson.
\end{proof}

\begin{lemma}\label{lem:h weakly increasing, B>=0}
Suppose $(\vec{b},\vec{e},\vec{a})$ is a triple of splitting types. For $1 \leq \mu\leq m$, let 
\[
h_\mu = \min\{\nu: 1 \leq \nu \leq n+1, S(\mu, \nu') < 0 \text{ for all } \nu' > \nu\}.
\] 
 For a fixed $\mu$, if $B(\mu,h_{\mu}) \geq 0$, then $h_{\mu+1}\geq h_{\mu}$.
\end{lemma}
\begin{proof}
    We have
    \[
    S(\mu+1,h_{\mu})=B(\mu,h_{\mu})+S(\mu,h_{\mu}) \geq 0,
    \]
    so $h_{\mu+1} \geq h_{\mu}$.
\end{proof}

\begin{theorem} \label{thm:main_realizability}
Suppose $(\vec{b},\vec{e},\vec{a})$ is a triple of splitting types. For $1 \leq \mu\leq m,$ let 
\[
h_\mu = \min\{\nu: 1 \leq \nu \leq n+1, S(\mu, \nu') < 0 \text{ for all } \nu' > \nu\}.
\] 
Then the following statements are equivalent: 
\begin{enumerate}
    \item $\trip$ is realizable;
    \item $\trip$ is weakly eligible, and the conditions of Lemma \ref{lem:necc} are satisfied;
    \item $\trip$ is weakly eligible, and $A(\mu,\nu), B(\mu,\nu) \geq 0$ whenever $\nu \geq h_{\mu}$.
\end{enumerate}
\end{theorem}
\begin{proof}
    Our argument above shows that (1) implies (2). We will now show that (2) implies (3). 

    Suppose (3) fails. Let $\mu_0$ be the smallest index such that for some $\nu_0\geq h_{\mu_0}$, $A(\mu_0,\nu_0) < 0$ or $B(\mu_0,\nu_0) < 0$. If $A(\mu_0,\nu_0)<0$, then letting \[\nu = \max \{1 \leq \nu' \leq n: A(\mu_0, \nu') < 0\},\]
    we have $\nu\geq \nu_0$. However, this means $\nu+1>h_{\mu_0}$ so by definition of $h_{\mu_0}$, we get $S(\mu_0,\nu+1)<0$, which fails condition (2). 
    
    On the other hand, if $B(\mu_0, \nu_0) < 0$, then let 
    \[\mu = \min \{1 \leq \mu' \leq m: B(\mu', \nu_0) < 0\}.\] Conditions of Lemma \ref{lem:necc} would then require that $S(\mu,\nu_0+1)\geq 0$, but this implies $h_\mu>\nu_0$ while $\nu_0\geq h_{\mu_0}$, implying $h_\mu>h_{\mu_0}$. In particular $\mu < \mu_0$. However, by Lemma \ref{lem:h weakly increasing, B>=0} applied to $\mu, \mu+1, \dots, \mu_0-1$, we find that $h_{\mu_0} \geq h_{\mu}$, which is a contradiction.

    Lastly, the proof of (3) implies (1) is via an explicit construction of a pair of matrices realizing $\trip$, and is done in the next subsection.
\end{proof}

\subsection{Explicit construction of a pair of matrices} \label{sec:construction_matrices}
Assuming Condition (3) in Theorem \ref{thm:main_realizability}, we shall now construct the kernel and quotient matrices that define the first and second maps in an exact sequence
\[
0 \to \cO(\vec{b}) \xrightarrow{C} \cO(\vec{e}) \xrightarrow{G} \cO(\vec{a}) \to 0.
\]
First we will construct the quotient matrix $G$ which we will argue is surjective. Then we shall construct a map $C: \cO(\vec{b}) \to \cO(\vec{e})$ which will be an injective map of vector bundles satisfying $GC = 0$. Lastly, since $\deg \vec{b} = \deg \vec{e} - \deg \vec{a}$, in fact $C = \ker G$ and the pair $(C, G)$ exhibits $\trip$ as realizable.

Let $n' = \max \{\nu : a_i=e_i\ \forall\  i\leq \nu\}$, and $m':= m + 1 - \min \{\mu: b_i=e_{n+i}\ \forall\ i\geq \mu\}$. Define $h_{\mu}$ for $1 \leq \mu \leq m$ as in Theorem \ref{thm:main_realizability}. Then
\begin{enumerate}
    \item $h_1 = n'+1$;
    \item $h_{m-m'} \leq n$ and $h_{\mu} = n+1$ for $\mu > m-m'$;
    \item $h_1 \leq h_2 \leq \cdots \leq h_{m}$.
\end{enumerate}
To see the first fact, note that
\[
h_1 = \max\{\nu: S(1, \nu) \geq 0\} = \max \{\nu: T(0, \nu-1) \leq 0\}.
\]
The latter inequality $T(0, \nu-1) \leq 0$ is equivalent to
\[
(a_1 + \cdots + a_{\nu-1}) - (e_1 + \cdots + e_{\nu-1}) \leq 0.
\]
By condition (3) of weak eligibility, in fact 
\[
h_1 = \max\{\nu: a_i = e_i \text{ for } i \leq \nu-1\} = n'+1.
\]
To see the second fact, note that $h_{\mu} = n+1$ is equivalent to $S(\mu, n+1) \geq 0$, i.e. 
\[
b_{\mu} + \cdots + b_m \geq e_{\mu+n} + \cdots + e_{m+n}.
\]
By condition (4) of weak eligibility, in fact
\begin{align*}
\min\{\mu: h_{\mu} = n+1\} &= \min\{\mu: h_{\mu'} = n+1 \text{ for } \mu' \geq \mu\} \\ &= \min \{\mu: b_i = e_{n+i} \text{ for } i \geq \mu\} \\ &= m-m'+1.
\end{align*}
The last fact about $h_{\mu}$'s forming a weakly increasing sequence follows from Lemma \ref{lem:h weakly increasing, B>=0}.

In the construction below, we assume all of these facts, and shall only make use of $h_1 \leq h_2 \leq \cdots \leq h_{m-m'}$. Our discussion above shows that the $h_{\mu}$'s that do not show up in the construction are in some sense superfluous.

\subsubsection{Construction of the quotient matrix.} First we shall define the quotient matrix $G$. After fixing a splitting of $\cO(\vec{a})$ and $\cO(\vec{e})$, we can specify $G$ by specifying the following compositions:
\begin{enumerate}
    \item $\cO(\vec{e}) \xrightarrow{G} \cO(\vec{a}) \twoheadrightarrow \cO(a_1, \dots, a_{n'})$;
    \item $\cO(\vec{e}) \xrightarrow{G} \cO(\vec{a}) \twoheadrightarrow \cO(a_{h_{k}}, \dots, a_{h_{k+1}})$ for $1 \leq k \leq m-m'$.
\end{enumerate}
The first composition is easy to describe. By definition, the first $n'$ entries of $\vec{e}$ and $\vec{a}$ are the same. So the first composition is simply
\[
\cO(\vec{e}) \twoheadrightarrow \cO(e_1, \dots, e_{n'}) \xrightarrow{\cong} \cO(a_1, \dots, a_{n'}). 
\]
Later we shall define certain matrices $G_k$ that have one more column than rows. Specifically, for $1 \leq k \leq m-m'$, we shall define maps \[G_k: \cO(e_{h_{k} + k - 1}, \dots, e_{h_{k + 1} + k}) \to \cO(a_{h_{k}}, \dots, a_{h_{k + 1}}).\]
The second type of composition is then defined to be 
\[
\cO(\vec{e}) \twoheadrightarrow \cO(e_{h_{k} + k - 1}, \dots, e_{h_{k + 1} + k}) \xrightarrow{G_k} \cO(a_{h_{k}}, \dots, a_{h_{k + 1}}).
\]

Explicitly in terms of matrices, let $G^-$ denote the $n'\times n'$ block in the upper left corner of $G$, and $G^+$ the $(n-n')\times (m+n-m'-n')$ block whose entries are those in rows $n' + 1$ through $n$ and columns $n' + 1$ through $m+n-m'$. Then the matrix $G^+$ is the result of pasting together $G_1, \dots, G_{m-m'}$, as illustrated in Figure \ref{fig:quotientM+}. The block $G^-$ is the $n'\times n'$ identity matrix, and all entries not in $G^-$ or $G^+$ are zero.
\[
G = 
\left[\begin{array}{ccc|ccccc|ccc}
     & & & & & & & & & &\\
     & G^- & & & & & & & & &\\
     & & & & & & & & & &\\
     \hline
     & & & & & & & & & &\\
     & & & & & G^+& & & & & \\
     & & & & & & & & & &
\end{array}\right]
\]

\begin{figure}
    \centering
    \includegraphics[width=0.8\linewidth]{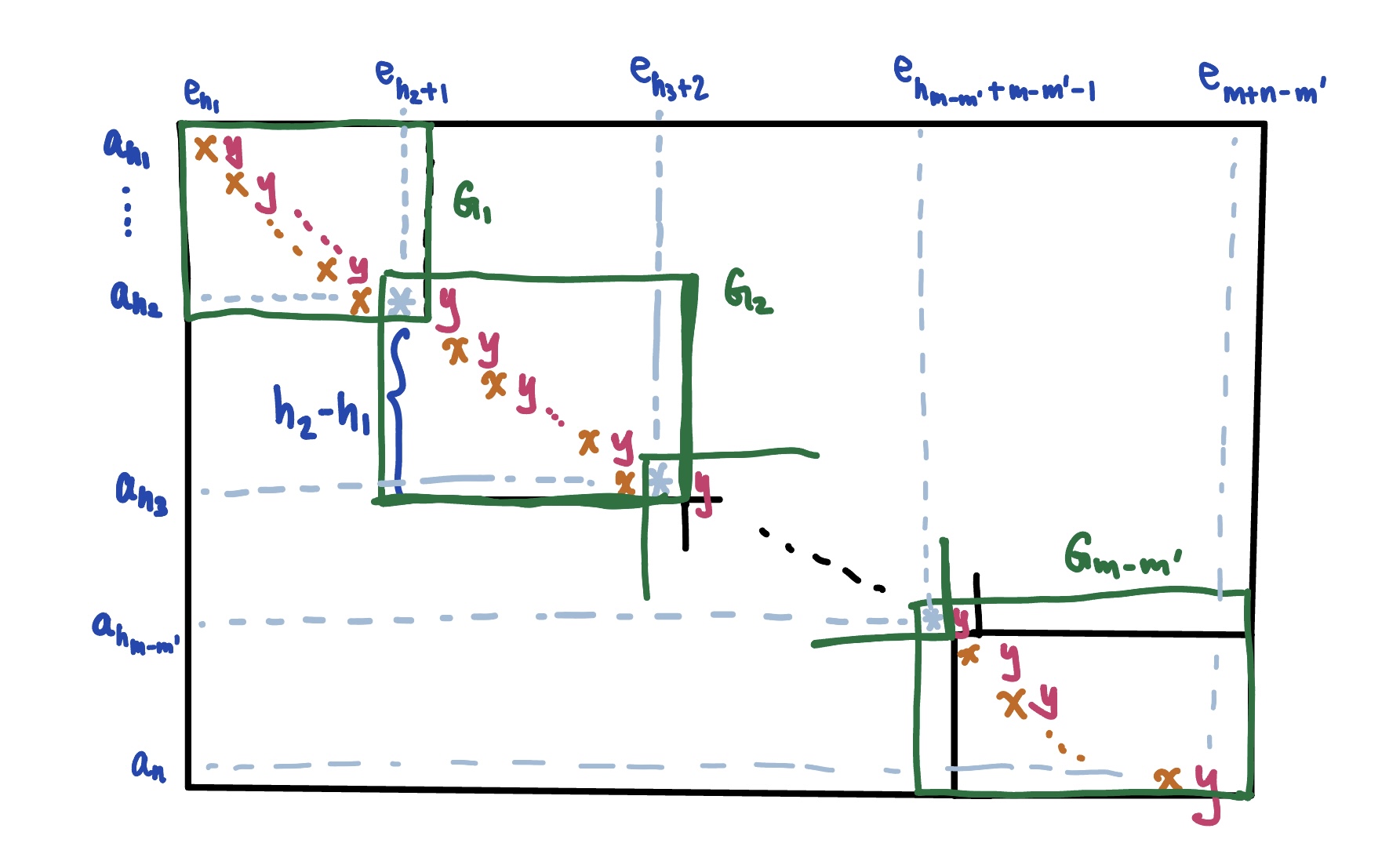}
    \caption{The matrix $G^+$ is obtained by gluing together $G_1, \dots, G_{m-m'}$ along the stars}
    \label{fig:quotientM+}
\end{figure}

Now let us define $G_k$ for $1 \leq k \leq m-m'$:
\begin{align*}
G_1 = 
\begin{bmatrix}
    x & y &         &  & & \\ 
      & x & y       &  & & \\ 
      &   & \ddots  & \ddots & & \\ 
      &   &   &  x & y & \\ 
      &   &        & & x & *_2
\end{bmatrix}, 
G_{m-m'} = 
\begin{bmatrix}
*_{m-m'} & y &         &  &   \\ 
      & x & y       &  &  \\ 
      &   & x  & y &  \\ 
      &   &  & \ddots & \ddots  \\ 
      &   &  & &  x & y 
\end{bmatrix}, 
\end{align*} and for $1 < k < m-m'$, 
\[
G_k = 
\begin{bmatrix}
*_{k} & y &         &  &  & \\ 
      & x & y       &  & & \\ 
      &   & x  & y & & \\ 
      &   &  & \ddots & \ddots & \\ 
      &   &  & &  x & y & \\ 
      &   &    &  &   & x & *_{k+1}
\end{bmatrix}.
\] 
In the matrix, we write $x, y$ to denote the pure power of $x$ and $y$ of the correct degree in each entry. For example the $(1,1)$-entry in $G_1$ is $x^{A(0, h_1)}$. We also use $*_k$ to denote the monomial \[*_k = x^{T(k-1, h_k)}y^{S(k, h_k)}.\] Note that the location of $*_k$ is the $(h_k, h_k+k-1)$-th entry of $G$. By definition of $h_k$, $S(k, h_k) \geq 0$, and $T(k-1, h_k) = -S(k, h_k+1) > 0$. So $*_k$ is indeed a monomial. Since
\[
T(k-1, h_k) + S(k, h_k) = -S(k, h_k+1) + S(k, h_k) = A(k-1, h_k),
\]
the degree of $*_k$ is correct. In order for the diagonal of $G_k$ to consist of nonzero entries, we need that $A(k-1, h_k + \ell) \geq 0$ for $0 \leq \ell \leq h_{k+1}-h_k$. This holds by our necessary conditions because $h_k + \ell \geq h_{k-1}$. Similarly, in order for the superdiagonal of $G_k$ to consist of nonzero entries, we need that $A(k, h_k + \ell) \geq 0$ for $0 \leq \ell \leq h_{k+1}-h_k$, which holds for the same reason.

Lastly, we claim that $G^+$ is surjective, which would imply that $G$ is surjective. We can consider the $(n-n') \times (n-n'+1)$-submatrix of $G^+$ consisting of those columns that do not contain a $*$-entry. Then this submatrix just has pure $x$-powers on the diagonal and pure $y$-powers on the superdiagonal, and in particular shows that the maximal minors of $G^+$ contain both a pure $x$-power and a pure $y$-power, and thus must define a surjective map of vector bundles on $\bP^1$.

\subsubsection{Construction of the kernel matrix.} 
After fixing a splitting of $\cO(\vec{e})$ and $\cO(\vec{b})$, we can specify the kernel matrix $C$ by specifying its restriction to each summand $\cO(b_k)$. If $k > m-m'$, then $b_k = e_{n+k}$ and the restriction of $C$ to $\cO(b_k)$ will just be the composition $\cO(b_k) \cong \cO(e_{n+k}) \to \cO(\vec{e})$. If $1 \leq k \leq m-m'$, the restriction of $C$ to $\cO(b_k)$ is defined as the composition 
\[
\cO(b_k) \xrightarrow{C_k} \cO(e_{h_k+k-1}, \dots, e_{h_{k+1}+k}) \to \cO(e_{n'+1}, \dots, e_{m+n-m'}),
\]
where $C_k$ is some map we shall define later satisfying the fact that $G_k C_k = 0$. 

Explicitly in matrix form, let $C^-$ denote the $m'\times m'$ block in the lower right corner, and $C^+$ the $(m+ n-m'-n')\times (m-m')$ block whose entries are those in rows $n'+1$ through $m+n-m'$ and columns $1$ through $m-m'$. Then the matrix $C^+$ is the result of pasting together $C_1, \dots, C_{m-m'}$, as illustrated in Figure \ref{fig:kernelC+}. The block $C^-$ is the $m'\times m'$ identity matrix, and all entries not in $C^-$ or $C^+$ are zero.
\[
C = 
\left[\begin{array}{ccc|ccc}
     & & & & & \\
     & & & & & \\
     \hline
     & & & & & \\
     & & & & & \\
     & C^+ & & & & \\
     & & & & & \\
     & & & & & \\
     \hline
     & & & & & \\
     & & & & C^- & \\
     & & & & &
\end{array}\right]
\]

\begin{figure}
    \centering
    \includegraphics[width=0.5\linewidth]{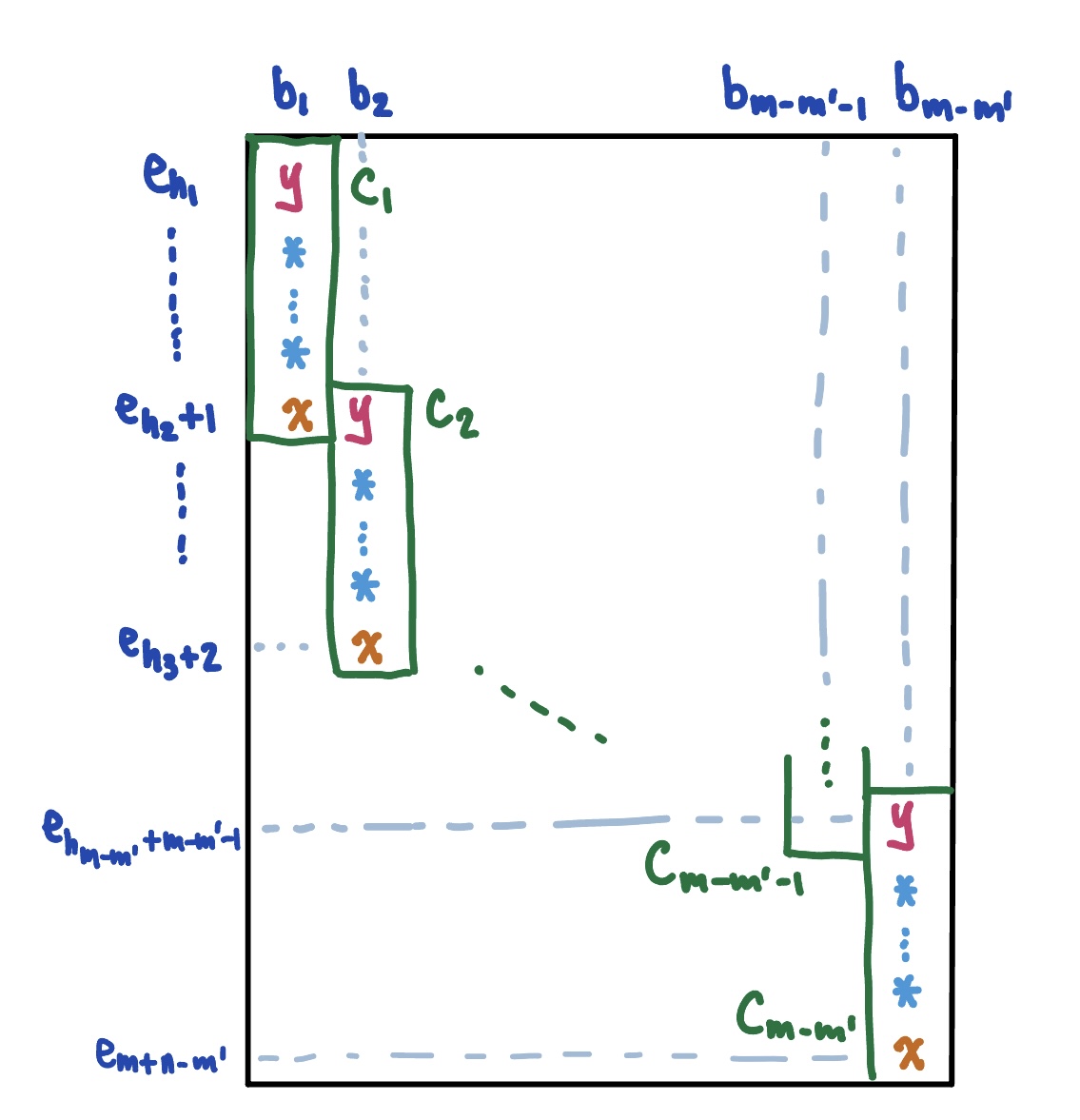}
    \caption{The matrix $C^+$ is obtained by gluing together $C_1, \dots, C_{m-m'}$ at the ends}
    \label{fig:kernelC+}
\end{figure}

We will now define the map $C_k: \cO(b_k) \to \cO(e_{h_k+k-1}, \dots, e_{h_{k+1}+k})$ satisfying $G_k C_k = 0$. 
Let 
\[
    C_k =\begin{bmatrix}
    y^{B(k, h_k)} \\
    -\star_{k, 1}\\
    \star_{k,2} \\ 
    \vdots \\
    \pm \star_{k, h_{k+1}-h_k}\\
    \mp x^{B(k, h_{k+1}+1)}
\end{bmatrix},
\]
where 
$\star_{k,\ell}$ denotes the monomial $x^{-S(k, h_k+\ell)}y^{-T(k, h_k+\ell-1)}$, and the signs are alternating going down the column. Note that all the entries of $C_k$ may be filled with a nonzero polynomial, since by our necessary conditions, $B(k, h_k + \ell) \geq 0$ whenever $\ell \geq 0$.

First, we check that 
\[
\deg\star_{k,\ell} = -S(k, h_k+\ell) - T(k, h_k+\ell-1) = B(k, h_k+\ell).
\]
This shows that all the entries are of the correct degree. Note that by definition of $h_k$, we have $-S(k, h_k + \ell) > 0$ for $1 \leq \ell \leq h_{k+1}-h_k$. By the necessary conditions, we know that $A(k, h_k + j) \geq 0$ for all $j \geq 0$, so 
\begin{align*}
-T(k, h_k + \ell - 1) &= S(k+1, h_k+\ell) \\ 
&= S(k+1, h_{k+1}) + A(k, h_k+\ell) + \cdots + A(k, h_{k+1}-1) \\ & \geq 0.
\end{align*} 
This shows that all the entries are indeed monomials. Lastly, we check that $G_k C_k = 0$ as claimed. It suffices to check the following three identities:
\begin{enumerate}
    \item $y^{B(k, h_k)} *_k - \star_{k,1} y^{A(k,h_k)} = 0$,
    \item $\star_{k,\ell} x^{A(k-1, h_k+\ell)} - \star_{k, \ell+1} y^{A(k, h_k+\ell)} = 0$, for $1 \leq \ell \leq h_{k+1}-h_k$,
    \item $\star_{k, h_{k+1}-h_k} x^{A(k-1, h_{k+1})} - x^{B(k, h_{k+1} + 1)} *_{k+1} = 0$.
\end{enumerate}
We explain how to check the second identity. The others are similar. 
On one hand, 
\[
    \star_{k,\ell} x^{A(k-1, h_k+\ell)} = x^{-S(k, h_k + \ell) + A(k-1,h_k + \ell)} y^{-T(k, h_k+\ell-1)},
\]
where $-S(k, h_k + \ell) + A(k-1,h_k + \ell) = -S(k, h_k+\ell+1)$.
On the other hand, 
\[
\star_{k, \ell+1} y^{A(k, h_k+\ell)} = 
x^{-S(k,h_k+\ell+1)}y^{-T(k,h_k+\ell)} y^{A(k, h_k+\ell)},
\]
where the $y$-degree can be rewritten as 
\[
-T(k,h_k+\ell) + A(k,h_k+\ell) = -T(k, h_k+\ell-1).
\]
Hence these two monomials are equal.

Additionally, similar to the case of the quotient matrix, we can check that $C^+$ is an injective map of vector bundles. Consider the submatrix of $C^+$ consisting of those rows that are the top and/or bottom of some $C_k$. Then this submatrix will also have two maximal minors which are respectively a pure $x$-power and a pure $y$-power. Hence $C$ is an injective map of vector bundles. Moreover, by construction and using the fact that $G_k C_k = 0$, we see that $GC = 0$. This finishes the construction.

\subsection{Realizability via balancing}
We will now prove another useful equivalent condition to realizability for a triple $\trip$. Let $[n] = \{1, 2, \dots, n\}$.

\begin{definition}[Balancing datum]
    A balancing datum for $\trip$ is a triple $(\sigma, \tau, \Gamma)$, where 
    $\sigma: [n] \to [m+n], \tau:[m] \to [m+n]$ are functions such that $\sigma \sqcup \tau: [n] \sqcup [m] \to [m+n]$ is a bijection. The last piece of data, $\Gamma$, is a function $\Gamma: [m] \times [n] \to \Z_{\geq 0}$, satisfying $\Gamma(i,j) \neq 0$ only if 
    \[ 
    b_i < e_{\tau(i)} \leq e_{\sigma(j)} < a_j,
    \]
    and if $\Gamma(i,j) \neq 0$, we have $\Gamma(i,j) \leq \min(e_{\tau(i)} - b_i, a_j - e_{\sigma(j)})$. Moreover, we require that
     \begin{align} \label{eq:balancing_datum_condition}
         \sum_{1 \leq j \leq n} \Gamma(i,j) = e_{\tau(i)} - b_i, \\ 
         \sum_{1 \leq i \leq m} \Gamma(i, j) = a_j - e_{\sigma(j)}.
     \end{align}
\end{definition}
Intuitively, $\Gamma$ gives instructions for how to balance $\vec{a}$ against $\vec{b}$ iteratively so that $a_j$ eventually becomes $e_{\sigma(j)}$, and $b_i$ becomes $e_{\tau(i)}$. 
\begin{example} 
    Let $\vec{e}=(2,7,8,11,20)$, $\vec{b}=(0,3,9)$, $\vec{a}= (13,23)$. Define $\sigma,\tau$ as follows:
    \begin{align*}
        \sigma(1)&=3\\
        \sigma(2)&=5\\
        \tau(1)&=1\\
        \tau(2)&=2\\
        \tau(3)&=4.
    \end{align*}
    Then a balancing datum built from $\sigma,\tau$ should give a means of balancing $\vec{b}$ against $\vec{a}$ until $\vec{b}=(2,7,11),\vec{a}=(8,20)$. Indeed, set 
    \begin{align*}
        \G(1,1)&=2\\
        \G(2,1)&=3\\
        \G(2,2)&=1\\
        \G(3,2)&=2.
    \end{align*}
    Then $(\gs, \tau,\G)$ is a balancing datum, and we can iteratively update the entries of $\vec{b}$ and $\vec{a}$ according to $\G$ as shown below:
    \begin{center}
\begin{tabular}{ c|c } 

$\vec{b}$ & $\vec{a}$ \\
\hline
$(0,3,9)$ & $(13,23)$ \\
$(2,3,9)$ & $(11,23)$ \\
$(2,6,9)$ & $(8,23)$ \\
$(2,7,9)$ & $(8,22)$ \\
$(2,7,11)$ & $(8,20)$ 
\end{tabular}
\end{center}
\end{example}

\begin{theorem} \label{thm:realzable_iff_balancing}
    A triple $\trip$ is realizable if and only if there exists a balancing datum for $\trip$.
\end{theorem}
\begin{proof}
    First we show that realizability implies the existence of a balancing datum. Suppose $(\vec{b},\vec{a})$ is realizable. Let 
\[ \tau(i) = \begin{cases} 
      h_i +i -1& i\leq m-m', \\
      n+i & m'<i\leq m, 
   \end{cases}
\]
where $h_{\mu}$ is as determined in Theorem \ref{thm:main_realizability}, and 
\[ \sigma(j) = j+\max\{\mu:h_\mu\leq j\}.
\]
In the construction of the kernel and quotient matrices in Subsection \ref{sec:construction_matrices}, the entry $(\tau(i), i)$ in the kernel matrix is filled, so is the entry $(j, \sigma(j))$ in the quotient matrix. Therefore, $b_{i} \leq e_{\tau(i)}$ for each $i$, and $a_{j}\geq e_{\sigma(j)}$ for each $j$. Moreover, by examining the quotient matrix, we see that $e_i$'s are ordered as follows: 
\begin{align*}
& e_{\sigma(1)} \leq \cdots \leq e_{\sigma(h_1-1)} 
\leq e_{\tau(1)} \\
\leq \ & e_{\sigma(h_1)} \leq \cdots \leq e_{\sigma(h_2-1)}
\leq e_{\tau(2)} \\ 
\leq \ & e_{\sigma(h_2)} \leq \cdots \leq e_{\sigma(h_3-1)}
\leq e_{\tau(3)} \\ 
\leq \ & \cdots \\  
\leq \ & e_{\sigma(h_{m-m'-1})} \leq \cdots \leq e_{\sigma(h_{m-m'}-1)} \leq e_{\tau(m-m')} \\ 
\leq \ & e_{\sigma(h_{m-m'})} \leq e_{\sigma(h_{m-m'} + 1)} \leq \cdots \leq e_{\sigma(n)}
\\
\leq \ & e_{\tau(m-m'+1)} \leq e_{\tau(m-m'+2)} \leq \cdots \leq e_{\tau(m)}.
\end{align*}
Note that $n' = h_1-1$. As such, by definition, $e_{\sigma(j)} - a_j = 0$ for $j \leq h_1-1$, and $e_{\tau(i)} - b_i = 0$ for $i \geq m-m'+1$.
Thus, we only need to define $\Gamma(i,j)$ whenever $h_1 \leq j \leq h_i-1$. We proceed inductively. Suppose that for each $h_1 \leq j < h_{\mu+1}$, we have assigned values for $\G(i,j)$ for all $i$, such that 
\begin{align*}
\sum_{1 \leq i \leq \mu} \Gamma(i, j) &= a_j - e_{\sigma(j)} \text{ when } 1 \leq j < h_{\mu+1}, \\ 
\Gamma(i,j) &= 0 \text{ for } i > \mu, \\ 
\sum_{h_1 \leq j < h_{\mu+1}} \Gamma(i, j) & \leq e_{\tau(i)} - b_i.
\end{align*}
Next we will define $\Gamma(i,j)$ where $h_{\mu+1} \leq j < h_{\mu+2}-1$. By realizability, we have that $S(\mu+2,h_{\mu+2})\geq 0$, i.e. 
\[
(a_1-e_{\sigma(1)})+\dots+(a_{h_{\mu+2}-1}-e_{\sigma(h_{\mu+2}-1)})\leq (e_{\tau(1)}-b_1)+\dots+(e_{\tau(\mu+1)}-b_{\mu+1}).
\]
In particular, 
\[
\sum_{j = h_{\mu+1}}^{h_{\mu+2}-1} (a_j - e_{\sigma(j)})
\leq 
\sum_{i = 1}^{\mu+1}\left (e_{\tau(i)} - b_i - \sum_{j = 1}^{h_{\mu+1}-1} \Gamma(i,j) \right ).
\]
This guarantees that there exists a way to assign $\Gamma(i,j)$ for $h_{\mu+1} \leq j < h_{\mu+2}-1$ while satisfying the inductive hypothesis. 

Note that after having assigned all the values, we have 
\begin{align*}
\sum_{j = 1}^n (a_j - e_{\sigma(j)})
&= 
\sum_{j = 1}^n \sum_{i = 1}^m \Gamma(i,j) \\
&= \sum_{i = 1}^m \sum_{j = 1}^n \Gamma(i,j) \\ 
&\leq \sum_{i = 1}^m (e_{\tau(i)} - b_i) \\ 
&= \sum_{j = 1}^n (a_j - e_{\sigma(j)}).
\end{align*}
Therefore, in fact $\sum_{j = 1}^n \Gamma(i,j) = e_{\tau(i)} - b_i$ for all $i$.

For the other direction, we use induction on $\sum_{(i,j)}\G(i,j)$. For the base case, if $\sum_{(i,j)}\G(i,j) = 0$, then $\trip$ is realizable via a split extension. Now suppose $(\gs,\tau,\G)$ is a balancing datum for $\trip$. By removing isomorphisms, without loss of generality, we may assume $b_i<e_\tau(i)$ for each $i$ and $a_j > e_{\sigma(j)}$ for all $j$. Now, let $i_0$ be maximal such that $\tau(i_0)<\gs(1)$. Then $\G(i, 1)>0$ for some $i\leq i_0$. Replacing $b_i$ with $b_i+1$, $a_1$ with $a_1-1$ and decreasing $\G(i,1)$ by $1$, gives a new triple and balancing datum, which is realizable by our inductive hypothesis. Undoing this replacement does not affect weak eligibility, and leaves $S(\mu, \nu)$ unchanged for all $(\mu,\nu)$ such that $h_\mu>1$ and $\nu\geq h_\mu$. $A(\mu,\nu)$ and $B(\mu,\nu)$ are already non-negative when $h_\mu = 1$ and will only be increased. Therefore, by Theorem \ref{thm:main_realizability}, $\trip$ is realizable. 
\end{proof}

\section{Stable and Strongly Stable Pairs}

Let $\cS, \cQ$ be the universal kernel and quotient bundles on $\QuotLF \times \bP^1$. Let $\vec{b}, \vec{a}$ be splitting types of rank $m, n$ and degree $\deg \cO(\vec{e})-d, d$ respectively. We write $\Sigma_{\vec{a}}(\cQ) \subseteq \QuotLF$ to denote the locally closed subscheme where the quotient bundle splits as $\vec{a}$. Similarly we write $\Sigma_{\vec{b}}(\cS) \subseteq \QuotLF$ to denote the locally closed subscheme where the kernel bundle splits as $\vec{b}$. We refer the reader to \cite{larson_degeneracy} and \cite[Section 2]{Lin2025} for background on splitting types and splitting loci.

\begin{proposition} \label{prop:strataQuotLF}
    Let $V_{\vec{b}} \subseteq \Hom(\cO(\vec{b}), \cO(\vec{e}))$ be the open subscheme consisting of injective maps with locally free cokernel, which admits a faithful action by $\Aut(\cO(\vec{b}))$ via precomposition. Then $\Sigma_{\vec{b}}(\cS) \cong [V_{\vec{b}}/\Aut(\cO(\vec{b}))]$. In particular, $\Sigma_{\vec{b}}(\cS)$ is smooth and irreducible of dimension $\hom(\cO(\vec{b}), \cO(\vec{e})) - \END(\cO(\vec{b}))$. Similarly, let $U_{\vec{a}} \subseteq \Hom(\cO(\vec{e}), \cO(\vec{a}))$ be the open subscheme consisting of surjective maps. Then $\Sigma_{\vec{a}}(\cQ) \cong [U_{\vec{a}} /\Aut(\cO(\vec{a}))]$ is smooth and irreducible of dimension $\hom(\cO(\vec{e}), \cO(\vec{a})) - \END(\cO(\vec{a}))$.
\end{proposition}
\begin{proof}
    One can check that $\Sigma_{\vec{b}}(\cS)$ and $[V_{\vec{b}}/\Aut(\cO(\vec{b}))]$ represent the same functor of points. Similarly for $\Sigma_{\vec{a}}(\cQ)$ and $[U_{\vec{a}}/\Aut(\cO(\vec{a}))]$.
\end{proof}
\begin{remark}
    Since $\QuotLF$ is in general not smooth, the map induced by $\cS$ from $\QuotLF$ to the stack of vector bundles of rank $m$ and degree $\deg \vec{e} - d$ is in general not smooth, and similarly for the map induced by $\cQ$. Thus it is slightly surprising that $\Sigma_{\vec{b}}(\cS)$ and $\Sigma_{\vec{a}}(\cQ)$ are smooth.
\end{remark} 

\begin{definition}[Stable Pairs] \label{def:stable_pairs}Let $\trip$ be realizable. Then a pair $(\vec{b}, \vec{a})$ is a \emph{stable pair} with respect to $\vec{e}$ if $\cO(\vec{a})$ is the generic cokernel among all maps in $\Sigma_{\vec{b}}(\cS)$, and if $\cO(\vec{b})$ is the generic kernel among all maps in $\Sigma_{\vec{a}}(\cQ)$. 
\end{definition}

\begin{definition}[Strongly Stable Pairs] Let $\trip$ be realizable. Then a pair $(\vec{b}, \vec{a})$ is a \emph{strongly stable pair} with respect to $\vec{e}$ if $(\vec{b}, \vec{a})$ is stable and $\dim \Hom(\cO(\vec{b}), \cO(\vec{a})) = \dim \Sigma_{\vec{a}}(\cQ) = \dim \Sigma_{\vec{b}}(\cS)$.    
\end{definition}

\begin{proposition} \label{prop:components_are_strongly_stable}
    Components of $\QuotLF$ are exactly the closure of $\Sigma_{\vec{b}}(\cS) \cap \Sigma_{\vec{a}}(\cQ)$, where $(\vec{b}, \vec{a})$ is a strongly stable pair.
\end{proposition}
\begin{proof}
    Let $Z$ be a component of $\QuotLF$. Let the generic splitting types for the universal kernel and quotient bundles on $Z \times \bP^1$ be $\vec{b}$ and $\vec{a}$ respectively. Then $\Sigma_{\vec{a}}(\cQ) \cap Z$ is dense in $Z$, which shows that $Z \subseteq \overline{\Sigma_{\vec{a}}(\cQ)}$. Since $Z$ is a component, in fact $Z = \overline{\Sigma_{\vec{a}}(\cQ)}$. Similarly, $Z = \overline{\Sigma_{\vec{b}}(\cS)}$. As such, $\Sigma_{\vec{a}}(\cQ) \cap \Sigma_{\vec{b}}(\cS)$ is dense in $Z$, and in particular dense in both $\Sigma_{\vec{a}}(\cQ)$ and $\Sigma_{\vec{b}}(\cS)$, which shows that $(\vec{b}, \vec{a})$ is a stable pair. To see that $(\vec{b}, \vec{a})$ is furthermore a strongly stable pair, recall that the tangent space to $\QuotLF$ of a closed point in $\Sigma_{\vec{a}}(\cQ) \cap \Sigma_{\vec{b}}(\cS)$ is exactly $\Hom(\cO(\vec{b}), \cO(\vec{a}))$. On the other hand, by Proposition \ref{prop:strataQuotLF} and our argument above, $\Sigma_{\vec{a}}(\cQ)$ is smooth and dense in $Z$. Therefore, the dimension of $\Hom(\cO(\vec{b}), \cO(\vec{a}))$ must equal the dimension of $\Sigma_{\vec{a}}(\cQ)$. Analogously, the dimension of $\Hom(\cO(\vec{b}), \cO(\vec{a}))$ must equal the dimension of $\Sigma_{\vec{b}}(\cS)$.

    Conversely, suppose that $(\vec{b}, \vec{a})$ is a strongly stable pair. Then $\Sigma_{\vec{b}}(\cS) \cap \Sigma_{\vec{a}}(\cQ)$ is nonempty since $\trip$ is realizable. The locus $\Sigma_{\vec{b}}(\cS) \cap \Sigma_{\vec{a}}(\cQ)$ is dense in $\Sigma_{\vec{b}}(\cS)$ and $\Sigma_{\vec{a}}(\cQ)$ since $(\vec{b}, \vec{a})$ is stable. In particular, the tangent space to $\Sigma_{\vec{b}}(\cS) \cap \Sigma_{\vec{a}}(\cQ)$ of a closed point has dimension equal to the dimensions of  $\Sigma_{\vec{b}}(\cS)$ and $\Sigma_{\vec{a}}(\cQ)$. Now the additional tangent space condition for being a strongly stable pair shows that the immersion of $\Sigma_{\vec{b}}(\cS) \cap \Sigma_{\vec{a}}(\cQ)$ in $\QuotLF$ is an open immersion.
\end{proof}

In the remainder of the section, we give explicit combinatorial characterizations of stable pairs and strongly stable pairs. 

First we define the notion of a combinatorially stable pair. We will soon show that this is equivalent to being stable, as in Definition \ref{def:stable_pairs}. 

\begin{definition}[Combinatorially stable pair] \label{def:combo_stable}
    A combinatorially stable pair is associated to the following package of data:
    \begin{enumerate}
        \item Nonnegative integers $m', n'$, such that $0 \leq m' \leq m, 0 \leq n' \leq n$;
        \item A partition $\Lambda=(P_1,Q_1,\dots,P_r,Q_r)$ of $(n'+1,\dots,m+n-m')$ into $2r$ parts such that $e_{P_{i,\text{last}}}<e_{Q_{i+1,1}}$ for each $1\leq i\leq r-1$. Given such data, we define 
        \begin{align*}
            \Delta P_i & = \sum_{1\leq j\leq |P_i|}(e_{P_{i,j}}-e_{P_{i,1}-1}-1)\\ 
            \Delta Q_i &= \sum_{1\leq j\leq |Q_i|}(e_{Q_{i,\last}+1}-e_{Q_{i,j}}-1). 
         \end{align*}
        
        \item A tuple $\vec{\delta}=(\delta_1,\dots,\delta_r)$, where $0 \leq \delta_i \leq \min(\Delta Q_i, \Delta P_i)$.
    \end{enumerate}
Given such data, let $\alpha(Q_i,\delta_i)$ be the most balanced tuple that is at least $e_{Q_i}$ in each entry and has $\sum_j(\alpha(Q_i,\delta_i)_j-e_{Q_{i,j}})=\delta_i$. Similarly, let $\beta(P_i,\delta_i)$ be the most balanced tuple that is at most $e_{P_i}$ in each entry and has $\sum_j(e_{P_{i,j}} -\beta(P_i,\delta_i)_j)=\delta_i$. Then the \emph{combinatorially stable pair} associated to the package above, which we denote by $\text{Pair}(m', n', \Lambda,\vec{\delta})$, is obtained by concatenating these balanced tuples as follows:
\begin{align*}
\vec{a} &= (e_1, \dots, e_{n'}, \alpha(Q_1, \delta_1), \dots, \alpha(Q_r, \delta_r)) \\ 
\vec{b} &= (\beta(P_1, \delta_1), \dots, \beta(P_r, \delta_r), e_{m+n-m'+1}, \dots, e_{m+n})
\end{align*}
\end{definition}
Note that in the above definition, $\alpha(Q_i,\delta_i)$ exists: it is the tuple obtained by adding $\delta_i$ to the entries of $Q_i$ in the most balanced manner, i.e. starting by adding to $Q_{i,1}$ until it becomes $Q_{i, 2}$, and then adding to $Q_{i,1}$ and $Q_{i,2}$ while keeping them balanced, until they reach $Q_{i,3}$, etc. Similarly, $\beta(P_i,\delta_i)$ can be obtained by subtracting $\delta_i$ from $P_i$ in the most balanced manner, starting with $P_{i, \last}$. 

We can read off a balancing datum for the combinatorially stable pair $\Pair(m', n', \Lambda,\vec{\delta})$ as follows: 
\begin{itemize}
    \item If $j$ is the index in $\vec{a}$ of $\alpha(Q_{s}, \delta_s)_k$, then $\gs(j)=Q_{s,k}$;
    \item If $i$ is the index in $\vec{b}$ of $\beta(P_{s}, \delta_s)_k$, then $\tau(i)=P_{s,k}$;
    \item $\G(i,j)>0$ only when $\gs(j)=Q_{s,k}$ and $\tau(i)=P_{s,\ell}$ for some $s,k,\ell$, and $\G$ sums to a total of $\delta_s$ over these pairs for a fixed $s$.
\end{itemize}

Even when the condition $\delta_i \leq \min(\Delta Q_i, \Delta P_i)$ does not hold, $\alpha(Q_i,\delta_i)$, $\beta(P_i,\delta_i)$, and $\text{Pair}(m', n', \Lambda,\vec{\delta})$ are still well defined. We will sometimes use these notations in this more general way.

Let us consider a few examples of combinatorially stable pairs.
\begin{example} \label{exmp:r=1}
    Consider the case where $r=1$ and $m'=n'=0$. Then we get a pair of the form 
    \begin{align*}
        \vec{b}&=\beta(P_1,\delta_1),\\
        \vec{a}&=\alpha(Q_1,\delta_1).
    \end{align*}
    By Theorem \ref{thm:realzable_iff_balancing}, this pair is realizable. We can directly check that $\vec{a}$ is the generic cokernel among all injective maps of vector bundles from $\cO(\vec{b}) \to \cO(\vec{e})$. For now, let $\pvec{a}^0$ be the generic cokernel. Note that we must have $a^0_n \geq e_{m+n}$, since $e_{m+n}$ must map nontrivially to at least one summand. If $a_n = \alpha(Q_1, \delta_1)_n > e_{m+n}$, then $\vec{a}$ is in fact balanced and we are done. If $a_n = e_{m+n}$, then $e_{m+n} \leq a^0_n \leq a_n = e_{m+n}$, so $a^0_n = e_{m+n}$. So $\pvec{a}^0$ can also be obtained by taking the generic cokernel among maps $\cO(\vec{b}) \to \cO(e_1, \dots, e_{m+n-1})$, and concatenating $\cO(e_{m+n})$. Let $Q_1' = Q_1 \setminus \{m+n\}$. By induction, the generic cokernel among maps $\cO(\vec{b}) \to \cO(e_1, \dots, e_{m+n-1})$ is $\alpha(Q_1', \delta_1)$, so $\pvec{a}^0 = (\alpha(Q_1',\delta_1), e_{m+n}) = \alpha(Q_1, \delta_1) = \vec{a}$ as desired. Dualizing and applying the same argument, we see that $\vec{b}$ is the generic kernel among all surjective maps from $\cO(\vec{e})$ to $\cO(\vec{a})$. Then $(\vec{b},\vec{a})$ is stable. 

    We can check also that $(\vec{b},\vec{a})$ is in fact strongly stable. Let $1 \leq j \leq n$ be the smallest index such that $a_i = e_{m+i}$ for all $i > j$ and $a_j < a_{j+1}$. (If $a_n > e_{m+n}$, then $j = n$.) Let $\vec{a} = (\vec{a}_-, \vec{a}_+)$, where $\vec{a}_- = (a_1, \dots, a_j)$, $\vec{a}_+ = (a_{j+1}, \dots, a_n)$. By construction, $\vec{e} = (\vec{e}_-, \vec{a}_+)$, where $\vec{e}_- = (e_1, \dots, e_{m+j})$. Moreover, it must be the case that $e_{m+j+1} > e_{m+j}$. Now,
    \begin{align*}
        \hom(\cO(\vec{e}), \cO(\vec{a})) - \END(\cO(\vec{a})) &= 
        \hom(\cO(\vec{e}_-), \cO(\vec{a}_-)) + \hom(\cO(\vec{e}_-), \cO(\vec{a}_+)) + 
        \hom(\cO(\vec{a}_+), \cO(\vec{a}_+)) \\
        & -   
        \hom(\cO(\vec{a}_-), \cO(\vec{a}_-))
        - 
        \hom(\cO(\vec{a}_-), \cO(\vec{a}_+))
        -
        \hom(\cO(\vec{a}_+), \cO(\vec{a}_+)) \\ 
        &= \hom(\cO(\vec{e}_-), \cO(\vec{a}_-)) - \hom(\cO(\vec{a}_-), \cO(\vec{a}_-)) \\ 
        & + \hom(\cO(\vec{e}_-), \cO(\vec{a}_+))
        - 
        \hom(\cO(\vec{a}_-), \cO(\vec{a}_+)).
    \end{align*}
    By construction, there is an exact sequence
    \[
    0 \to \oo(\vec{b}) \to \oo(\vec{e}_-) \to \oo(\vec{a}_-) \to 0.
    \]
    Since $\vec{a}_-$ is balanced, applying $\Hom(-, \cO(\vec{a}_-))$ to the sequence above, we find that
    \[
    0 \to \End(\cO(\vec{a}_-)) \to \Hom(\cO(\vec{e}_-), \cO(\vec{a}_-)) \to \Hom(\cO(\vec{b}), \cO(\vec{a}_-)) \to 0
    \]
    is exact.
    On the other hand, since $\ext^1(\cO(\vec{a}_-), \cO(\vec{a}_+))=0$, applying $\Hom(-, \cO(\vec{a}_+))$ to the sequence above, we find that
    \[
    0 \to \Hom(\cO(\vec{a}_-), \cO(\vec{a}_+)) \to \Hom(\cO(\vec{e}_-), \cO(\vec{a}_+)) \to \Hom(\cO(\vec{b}), \cO(\vec{a}_+)) \to 0
    \]
    is exact. Thus
    \begin{align*}
        \hom(\cO(\vec{e}), \cO(\vec{a})) - \END(\cO(\vec{a})) &= 
        \hom(\cO(\vec{b}), \cO(\vec{a}_-))+ 
        \hom(\cO(\vec{b}), \cO(\vec{a}_+)) \\ 
        &= \hom(\cO(\vec{b}), \cO(\vec{a})).
    \end{align*}
\end{example}

\begin{example} \label{ex:worked_out_stable_pairs}
    Let $\vec{e}=(0,4,5,6,8,12)$, $\rk\vec{a}=3$. Using Theorem \ref{thm:stable_pairs}, we can compute by hand the stable pairs for $\deg\vec{a}=20,\deg\vec{b} = 15$. Specifically, we enumerate all possibilities for the pair $(n',m')$. For each fixed $(n',m')$, we enumerate all possible partitions $\Lambda$, which place restrictions on possible degree distributions $\vec{\delta}$. 
    
    To illustrate this, consider the case $(n', m') = (1,0)$. If $\Lambda=((2,3),(4),(5),(6))$, we get $\delta_1+\delta_2=20-0-6-12=2$. We can compute $\Delta Q_1=1,\Delta P_2=1$, which forces $\delta_1 \leq 1$ and $\delta_2 \leq 1$, so we must have $\delta_1 = \delta_2 = 1$. This produces the stable pair
    \[\vec{b}=(4,4,7),\vec{a}=(0,7,13).\]
    If $\Lambda=((2),(3),(4,5),(6))$, then $\delta_1+\delta_2=3$, but $\Delta Q_1=0, \Delta P_2=2$, so no stable pairs arise this way.
    If $\Lambda=((2,3,4),\ (5,6))$, then $\vec{\delta}=(0,0)$ is forced and we get the pair
    \[\vec{b}=(4,5,6), \vec{a}=(0,8,12).\]
    All other partitions are not possible for degree reasons.
    
    After similar reasoning for all possible $(n', m')$ pairs, we get the following list of stable pairs. We also include the information of the dimension $D = \dim \Sigma_{\vec{b}}(\cS) = \dim \Sigma_{\vec{a}}(\cQ) = \hom(\cO(\vec{e}), \cO(\vec{a})) - \END(\cO(\vec{a}))$, as well as the dimension of the tangent space to any closed point of $\Sigma_{\vec{b}}(\cS) \cap \Sigma_{\vec{a}}(\cQ)$, which is $T = \hom(\cO(\vec{b}), \cO(\vec{a}))$.
    \[\vec{b}=(4,4,7),\vec{a}=(0,7,13), D = 35, T = 36,\]
    \[\vec{b}=(4,5,6), \vec{a}=(0,8,12), D = 36, T = 36,\]
    \[\vec{b}=(5,5,5),\vec{a}=(0,4,16), D = 36, T = 36,\]
    \[\vec{b}=(0,3,12),\vec{a}=(6,6,8), D = 37, T = 37,\]
    \[\vec{b}=(1,2,12),\vec{a}=(0,10,10), D = 38, T = 38,\]
    \[\vec{b}=(-5,8,12),\vec{a}=(6,7,7), D = 38, T = 38.\]
    When $D = T$, the pair is strongly stable. In this example, all pairs except the first one are strongly stable, which agrees with the criterion of Theorem \ref{thm:strongly_stable}.
\end{example}

We now proceed to prove that combinatorially stable pairs are exactly stable pairs.
\begin{definition}
We say that a balancing datum $(\gs,\tau,\G)$ \emph{minimal} if there is no other balancing datum 
with $\gs',\tau'$ 
such that $\tau'(i)\leq \tau(i)$ for all $1\leq i\leq m$ and at least one inequality is strict.
\end{definition}

\begin{lemma}\label{lem:first_block_stable}
    Suppose $(\vec{b},\vec{a})$ is a stable pair. Let $(\gs,\tau,\G)$ be a minimal balancing datum for $(\vec{b},\vec{a})$. Let $n'$ be maximal such that $\sigma(n')=n'\Llr a_{n'}=e_{n'}$. Let $m_1$ be maximal such that $\tau(m_1)=n'+m_1$, and $n_1$ maximal such that $\sigma(n'+n_1)=n'+m_1+n_1$. Then:
    \begin{enumerate}
        \item $a_{n'+n_1}<e_{\sigma(n'+n_1)+1}$; 
        \item $b_{m_1+1}>e_{\sigma(n'+n_1)}$; 
        \item $\G(i,j)=0$ when $i\leq m_1$ and $j>n'+n_1$. Equivalently, \[\sum_{i=1}^{m_1}(e_{\tau(i)}-b_i)=\sum_{j=1}^{n_1}(a_{n'+j}-e_{\gs(n'+j)}).\]
    \end{enumerate}
\end{lemma}

\begin{proof}
    (1)
    Suppose towards a contradiction that $a_{n'+n_1}\geq e_{\sigma(n'+n_1)+1}=e_{\tau(m_1+1)}$. We have $b_{m_1+1}\leq e_{\tau(m_1+1)}$ so consider two cases: either $b_{m_1+1}\leq e_{\gs(n'+n_1)}$ or $b_{m_1+1}>e_{\gs(n'+n_1)}$. 
    
    In the first case, we obtain another balancing datum by swapping $\tau(m_1+1)$ and $\gs(n'+n_1)$. 
    The value of $\Gamma$ needs to be adjusted accordingly. 
    The value of $\G$ needs to be decreased by a total of $e_{n'+m_1+n_1+1}-e_{n'+m_1+n_1}$ on pairs $(m_1+1,-)$ and similarly on pairs $(-, n'+n_1)$. To compensate and satisfy \eqref{eq:balancing_datum_condition}, $\G$ can be increased on pairs $(i, j)$, where $i, j$ are such that $\Gamma(i, n'+n_1)$ and $\Gamma(m_1+1, j)$ were decreased. This new balancing datum contradicts the minimality of the original balancing datum. 
    
    In the second case, $0<b_{m_1+1}-e_{\gs(n'+n_1)}\leq a_{n'+n_1}-e_{\gs(n'+n_1)}$. Choose $d_i<\G(i,n'+n_1)$ for each $i\leq m_1$ such that $\sum_{i\leq m_1}d_i=b_{m_1+1}-e_{\gs(n'+n_1)}$ and adding $d_i$ to $b_i$ for each $i\leq m_1$ does not change the order of these entries. Let $\pvec{b}'$ be the pair obtained from $\vec{b}$ by replacing $b_{m_1+1}$ with $e_{\gs(n'+n_1)}$ and $b_{i}$ with $b_i+d_i$ for each $i\leq m_1$. Then $\pvec{b}'>\vec{b}$. For a balancing datum for $(\pvec{b}',\vec{a})$:
    \begin{itemize}
        \item Start with the balancing datum given for $(\vec{b},\vec{a})$.
        \item Swap $\tau(m_1+1)$ and $\gs(n'+n_1)$.
        \item Decrease $\G(i,n'+n_1)$ by $d_i$ for each $i\leq m_1$.
        \item Decrease $\G$ by an additional total of $e_{n'+m_1+n_1+1}-b_{m_1+1}$ on pairs whose second index is $n'+n_1$; thus we will need to increase $\G$ by the same total on other pairs whose first indices were included in affected pairs.
        \item Decrease $\G$ by a total of $e_{n'+m_1+n_1+1}-b_{m_1+1}$ on entries with first coordinate $m_1+1$; we will need to increase $\G$ by exactly on this total on pairs whose second indices were included in affected pairs. We can do this by pairing these indices by those referred to in the above step.
    \end{itemize}
    Then $(\pvec{b}',\vec{a})$ is realizable, which contradicts the fact that $(\vec{b}, \vec{a})$ is stable.
    
    (2)
    Suppose towards a contradiction that $b_{m_1+1}\leq e_{\sigma(n'+n_1)}$. Let $j_0$ be minimal such that $\G(m_1+1,j_0)>0$. Since \[b_{m_1+1}\leq e_{\sigma(n'+n_1)}\leq a_{n'+n_1}\Lr e_{\tau(m_1+1)}-a_{n'+n_1}\leq e_{\tau(m_1+1)}-b_{m_1 + 1},\]
    we can choose $d_j$ for each $j\geq j_0$ such that $d_j\leq\G(m_1+1,j)$ and $\sum_{j\geq j_0}d_j=e_{\tau(m_1+1)}-a_{n'+n_1}$. Now let $\pvec{a}'$ be obtained from $\vec{a}$ by replacing $a_{n'+n_1}$ with $e_{\tau(m_1+1)}$ and $a_j$ with $a_j-d_j$ for each $j\geq j_0$. Let $N$ be the index of the entry obtained from $a_{n'+n_1}$, upon shifting its position by the minimal amount such that entries of $\pvec{a}'$ are in increasing order. $N$ will still be less than $j_0$ so $\pvec{a}'>\vec{a}$. For convenience we can henceforth assume $N=n'+n_1$ since this will not change the validity of a balancing datum. One can obtain a balancing datum for $(\vec{b},\pvec{a}')$ by setting $\tau(m_1+1)= n'+m_1+n_1$ and $\gs(n'+n_1)$ to what was previously $\tau(m_1+1)$, and decreasing $\G(m_1+1,j)$ by $d_j$ for $j\geq j_0$.
    
    (3) 
    By the conditions on a balancing datum, it is clear that
    \[\sum_{i=1}^{m_1}(e_{\tau(i)}-b_i)\geq\sum_{j=1}^{n_1}(a_{n'+j}-e_{\gs(n'+j)}).\]
    Now suppose this inequality is strict. Then there is some $i\leq m_1$ such that $\G(i,j)>0$ for some $j> n'+n_1$, and in particular $b_i<e_{n'+i}\leq e_{\gs(n'+n_1)}<b_{m_1+1}$ by part (2). Note that we will have $\tau(m_1+1)=\gs(n'+n_1)+1$ by construction of $n_1$. Define $\pvec{b}'>\vec{b}$ from $\vec{b}$ by increasing $b_i$ by 1 and decreasing $b_{m_1+1}$ by 1. For a balancing datum for $\pvec{b}',\vec{a}$, start with the balancing datum given for $(\vec{b},\vec{a})$, decrease $\G(i,j)$ by 1 and increase $\G(m_1+1,j)$ by 1. Then $(\pvec{b}',\vec{a})$ is realizable, which is a contradiction.
\end{proof}

\begin{theorem} \label{thm:stable_pairs}
    Fix $\vec{e}$. A pair $(\vec{b},\vec{a})$ is stable if and only if it is combinatorially stable. 
\end{theorem}

\begin{proof}
    First suppose $(\vec{b},\vec{a})$ is stable. Define $n', m_1, n_1$ as in the statement of Lemma \ref{lem:first_block_stable}. Then Lemma \ref{lem:first_block_stable} shows that a minimal balancing datum for $(\vec{b},\vec{a})$ will have $\G(i,j)\ne 0$ only if $1\leq i\leq m_1$ and $1\leq j\leq n'+n_1$ or $i$ and $j$ are both outside these ranges. Let $(\pvec{b}',\pvec{e}',\pvec{a}')$ be defined by removing the first $m_1$ entries from $\vec{b}$, removing the first $n'+m_1+n_1$ entries from $\vec{e}$ and removing the first $n'+n_1$ from $\vec{a}$. Then if $(\pvec{b}',\pvec{a}')$ is not stable with respect to $\pvec{e}'$, one can combine the balancing datum of a generization with $\G(i,j)$ for $1\leq i\leq m_1$ and $1\leq j\leq n'+n_1$ to get a balancing datum for a generization of $(\vec{b},\vec{a})$, so $(\pvec{b}',\pvec{e}',\pvec{a}')$ must be stable and thus combinatorially stable by inductive hypothesis. 

    Let $(\vec{b},\vec{a}) = \Pair(m', n', \Lambda,\vec{\delta})$ be a combinatorially stable pair. To see that $(\vec{b},\vec{a})$ is stable, we assume without loss of generality that $m' = n' = 0$, and proceed by induction on $r$. The base case $r=1$ is explained in Example \ref{exmp:r=1}. By symmetry, it suffices to show that $\vec{a}$ is the generic cokernel among all maps from $\cO(\vec{b})$ to $\cO(\vec{e})$. Let $(\pvec{b}', \pvec{e}', \pvec{a}')$ be the triple obtained by removing the $(P_r, Q_r)$ block. By induction, this triple is stable. Let $\vec{a}_{\bal}$ be the most balanced cokernel for a general map from $\cO(\vec{b})$ to $\cO(\vec{e})$. 
    
    Then we have the following diagram of short exact sequences. Let us unpack this diagram. We start with a generic map from $\cO(\vec{b}) \to \cO(\vec{e})$, whose cokernel is $\cO(\vec{a}_{\bal})$. This constitutes the middle row. Then by construction, $\cO(\vec{b})$ admits a distinguished inclusion from $\cO(\beta(P_r, \delta_r))$, which has cokernel $\cO(\pvec{b}')$. Similarly, $\cO(\vec{e})$ admits a distinguished inclusion from $\cO(\vec{e}_{P_r,Q_r})$, which has cokernel $\cO(\pvec{e}')$. The injective map $\cO(\vec{b}) \to \cO(\vec{e})$ restricts to an injective map $\cO(\beta(P_r,\delta_r)) \to \cO(\vec{e}_{P_r,Q_r})$, which is generic among such maps, which allows us to conclude that the cokernel is exactly $\cO(\alpha(Q_r,\delta_r))$, by the base case of the induction. Similarly, the injective map $\cO(\vec{b}) \to \cO(\vec{e})$ also restricts to a generic map $\cO(\pvec{b}') \to \cO(\pvec{e}')$, which implies that it is injective with cokernel $\cO(\pvec{a}')$. By the snake lemma, this shows that the induced map $\cO(\alpha(Q_r,\delta_r)) \to \cO(\vec{a}_{\bal})$ is injective.
    \[
\begin{tikzcd}
0 \arrow[r] & \cO(\pvec{b}')  \arrow[r]                               & \cO(\pvec{e}') \arrow[r]                      & \cO(\pvec{a}') \arrow[r]                          & 0 \\
0 \arrow[r] & \cO(\vec{b}) \arrow[u] \arrow[r]               & \cO(\vec{e}) \arrow[u] \arrow[r]              & \cO(\vec{a}_{\bal}) \arrow[u] \arrow[r]                  & 0 \\
0 \arrow[r] & {\cO(\beta(P_r,\delta_r))} \arrow[u] \arrow[r] & {\cO(\vec{e}_{P_r, Q_r})} \arrow[u] \arrow[r] & {\cO(\alpha(Q_{r},\delta_r))} \arrow[u] \arrow[r] & 0
\end{tikzcd}
    \]
    By construction, the last entry of $\pvec{a}'$ must be strictly less than the first entry of $\alpha(Q_r,\delta_r)$, which shows that the last vertical short exact sequence is in fact split, proving that $\vec{a}_{\bal} = \vec{a}$. 
\end{proof}

Among stable pairs, the condition for being strongly stable can also be combinatorially characterized.

\begin{lemma}\label{lem:ineq1}
For any tuples $\vec{a}$, $\pvec{a}'$, $\pvec{a}''$ with $\pvec{a}'\geq\pvec{a}''$, we have:
\begin{enumerate}
    \item $\hom(\oo(\vec{a}),\oo(\pvec{a}'))\leq \hom(\oo(\vec{a}),\oo(\pvec{a}''));$
    \item $\hom(\oo(\pvec{a}'),\oo(\pvec{a}))\leq \hom(\oo(\pvec{a}''),\oo(\pvec{a})).$
\end{enumerate}
\end{lemma}

\begin{proof}
There exists a degeneration of $\cO(\vec{a}')$ to $\cO(\vec{a}'')$ over $\mathbb{A}^1$ by \cite[Proposition 1]{Lin2025}. Let $E$ be the total family of vector bundles over $\bA^1$, so that $\cO(\vec{a})^{\vee}\otimes E$ and $E^{\vee} \otimes \cO(\vec{a})$ are flat over $\mathbb{A}^1$. By semicontinuity, $h^0$ can only increase or stay constant over the central fiber. 
\end{proof}

\begin{lemma} \label{lem:ineq2}
For any tuples $\vec{a}$, and $\pvec{a}'$ with $\vec{a}>\pvec{a}'$, we have \[\hom(\oo(\vec{a}),\oo(\pvec{a}'))-\END(\oo(\pvec{a}')<0.\]
\end{lemma}
\begin{proof}
Let $\pvec{a}''=(a'_1,\dots,a'_n)$. Since $\vec{a}>\pvec{a}'$ there is some splitting type \[\pvec{a}''=(a'_1,\dots,a'_i+1, \dots, a'_j-1, \dots,a'_n)\]
such that $\pvec{a} > \pvec{a}'' > \pvec{a}'$, so without loss of generality, assume $\vec{a}$ is of this form. Then 
\begin{align*}
    \hom(\oo(\vec{a}),\oo(\pvec{a}'))-\END(\oo(\pvec{a}')) 
    & =\hom(\oo(a'_i+1, a'_j-1),\oo(\pvec{a}'))-\hom(\oo(a'_i, a'_j),\oo(\pvec{a}')) \\ 
    &= -\#\{a_{k}': a_{k}' \geq a_i'\} + \#\{a_k': a_k' \geq a_{j}'-1\}.
\end{align*}
Since $a_j'-1 \geq a_i' + 1$, we have $\#\{a_{k}': a_{k}' \geq a_i'\} > \#\{a_k': a_k' \geq a_{j}'-1\}$, so $\hom(\oo(\vec{a}),\oo(\pvec{a}')) < \END(\oo(\pvec{a}'))$ as desired. 
\end{proof}

\begin{lemma} \label{lem:greater_pair}
Suppose $(\vec{b},\vec{a})$ and $(\pvec{b}',\pvec{a}')$ are realizable pairs and $\vec{a}\geq\pvec{a}'$ and $\vec{b}\geq\pvec{b}'$. If either inequality is strict, then $(\pvec{b}',\pvec{a}')$ is not strongly stable. 
\end{lemma}

\begin{proof}
If exactly one of the inequalities is strict, then in fact $(\pvec{b}',\pvec{a}')$ is not stable. So it remains to consider the case where both inequalities are strict. 

Let $\text{sort}(\vec{b}, \vec{a})$ denote the splitting type obtained by concatenating $\vec{b}, \vec{a}$, and sorting it into a weakly increasing sequence. By \cite[Lemma 3]{Lin2025}, $\text{sort}(\vec{b},\vec{a})\leq \vec{e}$, so we have 
\begin{align*}
    \dim U_{\pvec{a}'}-\dim T_{\pvec{b}',\pvec{a}'}&= \hom(\oo(\vec{e}),\oo(\pvec{a}'))-\END(\oo(\pvec{a}'))-\hom(\oo(\pvec{b}'),\oo(\pvec{a}'))\\
     &\leq \hom(\oo(\text{sort}(\vec{b},\vec{a})),\oo(\pvec{a}'))-\END(\oo(\pvec{a}'))-\hom(\oo(\pvec{b}'),\oo(\pvec{a}'))\\
     &\leq \hom(\oo(\vec{a}),\oo(\pvec{a}')) -\END(\oo(\pvec{a}')),
\end{align*}
where the first inequality is a consequence of Lemma \ref{lem:ineq1}. By Lemma \ref{lem:ineq2}, $\dim U_{\pvec{a}'} < \dim T_{\pvec{b}',\pvec{a}'}$, so $(\pvec{b}',\pvec{a}')$ is not strongly stable.
\end{proof}

\begin{lemma}\label{lem:block_overlapping_fails_strongly}
Let $\vec{a}_-=(a_1,\dots,a_{k}),\vec{a}_+=(a_{k+1},\dots,a_n)$, $\vec{b}_-=(b_1,\dots,b_{\ell}),\vec{b}_+=(b_{\ell+1},\dots,b_m)$. Suppose $(\vec{b}_-,\vec{a}_-)$ is a stable pair with respect to $\vec{e}_{-}:=(e_1,\dots,e_{k+\ell})$, and $(\vec{b}_+,\vec{a}_+)$ is stable with respect to $\vec{e}_+:=(e_{k+\ell+1},\dots,e_{m+n})$. Further, assume that the greatest entry of $\pvec{a}_-$ is less than the least entry of $\pvec{e}_+$. Then $(\vec{b},\vec{a})$ is strongly stable with respect to $\vec{e}$ if and only if it is stable, $(\vec{b}_-,\vec{a}_-)$ and $(\vec{b}_+,\vec{a}_+)$ are strongly stable with respect to $\vec{e}_-$ and $\vec{e}_+$, respectively, and $a_{-,\text{last}}<b_{+,1}$ for all $i,j$. 
\end{lemma}

\begin{proof}
Since we assumed stability, $(\vec{b},\vec{a})$ is strongly stable if and only if the tangent space has the correct dimension, namely
\[\hom(\oo(\vec{e}),\oo(\vec{a}))-\END(\oo(\vec{a}))-\hom(\oo(\vec{b}),\oo(\vec{a}))=0.\]
However, since the top entry of $\pvec{a}_-$ is less than the smallest entry of $\pvec{e}_+$, \[\hom\left(\oo(\pvec{e}_+),\oo(\pvec{a}_-)\right)=0,\]
and similarly 
\[\hom\left(\oo(\pvec{a}_+),\oo(\pvec{a}_-)\right)=0.\]
So we have
\begin{align*}
\hom(\oo(\vec{e}),\oo(\vec{a}))&=\hom(\oo(\vec{e}_-),\oo(\vec{a}_-))+\hom(\oo(\vec{e}_-),\oo(\vec{a}_+))+\hom(\oo(\vec{e}_+),\oo(\vec{a}_+)), \\
\END(\oo(\vec{a}))&=\END(\oo(\vec{a}_-))+\hom(\oo(\vec{a}_-),\oo(\vec{a}_+))+\END(\oo(\vec{a}_+)),\\
\hom(\oo(\vec{b}),\oo(\vec{a}))&=\hom(\oo(\vec{b}_-),\oo(\vec{a}_-))+\hom(\oo(\vec{b}_-),\oo(\vec{a}_+)) \\ &+\hom(\oo(\vec{b}_+),\oo(\vec{a}_-))+\hom(\oo(\vec{b}_+),\oo(\vec{a}_+)).
\end{align*}
Therefore
\begin{align*}
    & \hom(\oo(\vec{e}),\oo(\vec{a}))-\END(\oo(\vec{a}))-\hom(\oo(\vec{b}),\oo(\vec{a})) \\
    = \ & \left(\hom(\oo(\vec{e}_-),\oo(\vec{a}_-))-\END(\oo(\vec{a}_-))-\hom(\oo(\vec{b}_-),\oo(\vec{a}_-))\right) \\ 
     + \ & \left(\hom(\oo(\vec{e}_+),\oo(\vec{a}_+))-\END(\oo(\vec{a}_+))- \hom(\oo(\vec{b}_+),\oo(\vec{a}_+))\right)\\
    + \ & \left(-\hom(\oo(\vec{b}_+),\oo(\vec{a}_-))\right).
\end{align*}
Note that each of the three terms above is nonpositive, so the left hand side is zero if and only if each term on the right is zero. This happens if and only if $(\vec{b}_-,\vec{a}_-)$ and $(\vec{b}_+,\vec{a}_+)$ are strongly stable with respect to $\vec{e}_-$ and $\vec{e}_+$, respectively, and $a_{-,i}<b_{+,j}$ for all $i,j$.
\end{proof}

\begin{theorem} \label{thm:strongly_stable}
Let $(\vec{b}, \vec{a})$ be a stable pair, associated to the combinatorial data of $(m', n', \Lambda)$ as in Definition \ref{def:stable_pairs}, so that
\begin{align*}
\vec{a} &= (e_1, \dots, e_{n'}, \alpha(Q_1, \delta_1), \dots, \alpha(Q_r, \delta_r)) \\ 
\vec{b} &= (\beta(P_1, \delta_1), \dots, \beta(P_r, \delta_r), e_{m+n-m'+1}, \dots, e_{m+n})
\end{align*}
Then $(\vec{b}, \vec{a})$ is strongly stable if and only if 
\begin{align*}
    e_{n'} & < \beta(P_1, \delta_1)_1, \\ 
    \alpha(Q_i, \delta_i)_{\last} & < \beta(P_{i+1}, \delta_{i+1})_1 \text{ for } 1 \leq i < r,\\ 
    \alpha(Q_r, \delta_r)_{\last}& < e_{m+n-m'+1}. 
\end{align*} 

\end{theorem}

\begin{proof}
    We use induction on $r$. As a base case, take $r=1$ and $m'=n'=0$, which gives a strongly stable pair as shown in Example \ref{exmp:r=1}. For $r=1$ and $n'=0$, we can let $\vec{b}_-=\beta(P_1,\delta_1)$, $\vec{a}_-=\alpha(Q_1,\delta_1)$, $\vec{b}_+=(b_{m-m'+1},\dots,b_m)$, $\vec{a}_+=()$, and apply Lemma \ref{lem:block_overlapping_fails_strongly}. In the general case of $r=1$, we can now apply Lemma \ref{lem:block_overlapping_fails_strongly}, setting $\vec{b}_-=()$, $\vec{a}_-=(a_1,\dots, a_{n'})$. For $r>1$, apply Lemma \ref{lem:block_overlapping_fails_strongly} to $\vec{b}_-=\beta(P_1,\delta_1)$, $\vec{a}_-=(a_1,\dots,a_{n'},\alpha(Q_1,\delta_1))$.
\end{proof}

\section{Irreducibility}
\begin{lemma}\label{lem:most_balanced}
Fix $d$, and let $d'=\deg\vec{e}-d$. Let
\begin{align*}
    f(j)&:=d-(e_1+\dots+e_{j}+e_{j+2}+\dots+e_{n+1}),\\
    g(i)&:=e_{n}+\dots+e_{m+n-i-1}+e_{m+n-i+1}+\dots+e_{m+n}-d',
\end{align*}
and $n',m'\geq 0$ minimal such that $f(n'),g(m')\geq 0$ with the inequality strict when $e_{n'+1}=e_{n'+2}$ (resp. $e_{m+n-m'}=e_{m+n-m'+1}$). Then:
\begin{enumerate}
    \item the most balanced tuple $\vec{a}$ of degree $d$ such that there exists a surjection $\oo(\vec{e})\to\oo(\vec{a})$ is unique, and given by
    \[\vec{a}= \left(e_1,\dots,e_{n'},\alpha((n'+2,\dots,n+1), f(n'))\right);\]
    \item the most balanced tuple $\vec{b}$ of degree $d'$ such that there exists an injective map with locally free cokernel $\oo(\vec{b})\to \oo(\vec{e})$ is unique and given by
    \[\vec{b}= \left(\beta((n, \dots,m+n-m'-1), g(m')), e_{m+n-m'+1},\dots,e_{m+n}\right)\]
\end{enumerate}
\end{lemma}
\begin{proof}
    We will prove (1) via Proposition \ref{prop:onlyif}. The argument for (2) is analogous. Suppose $\pvec{a}'$ is the splitting type of some degree $d$ quotient of $\oo(\vec{e})$. Then $a'_i=e_i$ for $i\leq n'$, since otherwise, $\pvec{a}'$ would have degree at least $(e_1+\dots+e_{n'-1}+e_{n'+1}+\dots+e_{n+1})>d$. By definition, $\vec{a}$ is the most balanced tuple admitting a surjection from $\cO(\vec{e})$ such that $a_i = e_i$ for $i \leq n'$ and $a_{n'+1} > e_{n'+1}$. Let $k > n'$, and let 
    \[
    \pvec{a}' = (e_1, \dots, e_{k}, \alpha((k+2, \dots, n+1), f(k)).
    \]
    It remains to check that $\vec{a} \geq \pvec{a}'$. For $1 \leq \ell \leq k$, it is clear that 
    \[
    a_1 + \cdots + a_{\ell} \geq a_1' + \cdots + a'_{\ell}.
    \]
    On the other hand, for $k < \ell \leq n$, to check that the condition above,
    it is equivalent to check 
    \[
    a_{\ell+1} + \cdots + a_n \leq a_{\ell+1}' + \cdots + a'_{n}.
    \]
    For this, it suffices to observe that the last $n-k$ entries of $\vec{a}$ can be written as $\alpha((k+2, \dots, n+1), \delta)$, where $\delta \leq f(n') \leq f(k)$.    
\end{proof}

\begin{theorem}
Let $d=\deg\vec{a}$, $d'=\deg\vec{b}$, $\Delta = d-(e_1+\dots+e_{n'}+e_{n'+m-m'+1}+\dots+e_{m+n-m'})$. The following are equivalent:
\begin{enumerate}
    \item $\QuotLF$ is irreducible;
    \item $(\vec{b},\vec{a})$ is realizable, where $\vec{a}$ and $\vec{b}$ are as given in Lemma \ref{lem:most_balanced};
    \item we have 
    \begin{align*}
        \alpha((n'+m-m'+1,\dots,m+n-m'),\Delta)&=\alpha((n'+2,\dots,n+1), f(n'))<e_{m+n-m'+1},\\
        \beta((n'+1,\dots,n'+m-m'),\Delta)&= \beta((n, \dots,m+n-m'-1), g(m'))> e_{n'},
    \end{align*}
where the inequalities are checked on entries. 
\end{enumerate}
\end{theorem}

\begin{proof} 
By Lemma \ref{lem:greater_pair}, when $(\vec{b},\vec{a})$ is realizable, it is the unique strongly stable pair. Otherwise, $\vec{b}$ and $\vec{a}$ are realized in distinct strongly stable pairs, in which case $\QuotLF$ has at least two irreducible components. Hence $(1)\Llr (2)$. 

To see that $(3) \Rightarrow (2)$, note that the condition for (3) is exactly saying that the unique partition with $r = 1$ and $\vec{\delta} = (\Delta)$ exhibits $(\vec{b}, \vec{a})$ as stable, and in particular realizable. We now show $(2) \Rightarrow (3)$. As above, if $(\vec{b},\vec{a})$ is realizable, it must be strongly stable, and therefore satisfies the conditions of Theorem \ref{thm:strongly_stable}. It suffices to show that in a minimal balancing datum exhibiting $(\vec{b},\vec{a})$ as strongly stable, the partition has $r = 1$. Note that by construction, $a_i = e_i$ for $i \leq n'$, and $a_{n'+1} > e_{n'+1}$ (and similarly for $m'$ and $\vec{b}$).

First suppose that the tuple $(e_{n'+1}, \dots, e_{m+n-m'})$ is balanced, i.e. $e_{n'+1}\geq e_{m+n-m'}-1$. If $e_{n'+1}= e_{m+n-m'}$, then a minimal balancing datum for $(\vec{b},\vec{a})$ would be associated to a partition with $r=1$. If $e_{n'+1}=e_{m+n-m'}-1$ then let $n'+1\leq k\leq {m+n-m'-1}$ such that $e_k+1= e_{k+1}$. In this case, if a minimal balancing datum for $(\vec{b},\vec{a})$ is associated to a partition with $r>1$, then we must have $r = 2$ and the break between $Q_1$ and $P_2$ must occur between $e_k$ and $e_{k+1}$ by the combinatorial conditions for stability. In fact the bound on $\delta_i$ forces $\delta_1 = \delta_2 = 0$, 
but this contradicts the definition of $n'$.

Now suppose $e_{n'+1}< e_{m+n-m'}-1$ and that the entries of $\vec{e}$ are increasing at at least two points in this range or that $\delta_1>0$; otherwise, if $r=2$, we can argue as in the preceding paragraph. Then $e_k<e_{k+1}$ for some smallest $n'+1\leq k \leq {m+n-m'-1}$. Without loss of generality, take $k< n$. Then a balancing datum will have $\tau(1) = n'+1 \leq k$. Since $e_k < e_{n}$ or $\delta_1>0$, $b_1<e_n$, which will mean $\beta((n-1, \dots,m+n-m'-1), g(m'))$ is balanced with first entry less than or equal to $e_k$. Similarly, $\gs(n)=m+n-m'$, which will require $\alpha((n'+2,\dots,n+1), f(n'))$ to be balanced. Since $e_{m+n-m'}-e_k>1$, the only way this is possible is by having $r=1$. 
\end{proof}

Notably, this implies $\QuotLF$ is irreducible in the following cases:
\begin{enumerate}
    \item $d=e_1+\dots+e_n$;
    \item $d\geq n(e_{m+n}-1)+1$ and $d'\leq m(e_1+1)-1$.   
\end{enumerate}
The latter is Corollary \ref{cor:intro_irreducibility}.

However, sometimes we can find degrees in between these bounds where irreducibility still occurs.

\begin{example}
    If $\vec{e}=(1,7,8,9,20)$, then $\text{Quot}^{3,20}_{\pp^1}(\oo(\vec{e}))^\circ$ is irreducible, and the unique component corresponds to the strongly stable pair $\vec{b}=(5,20),\vec{a}=(1,9,10)$.
\end{example}

\section{Connectedness} 
In this section, we prove Theorem \ref{thm:intro_connected}.

We say two realizable pairs $(\vec{b}, \vec{a})$ and $(\pvec{b}', \pvec{a}')$ are \textit{connected} if the strata $\Sigma_{\pvec{b}'}(\cS) \cap \Sigma_{\pvec{a}'}(\cQ)$ and $\Sigma_{\pvec{b}}(\cS) \cap \Sigma_{\pvec{a}}(\cQ)$ are in the same connected component of $\QuotLF$. 

Our general strategy for the proof will be to take a given stable pair and find another stable pair with a more balanced quotient splitting type such that the two pairs are connected. To this end, we will make repeated use of the following:

\begin{lemma} \label{lem:connecting_pairs}
    Suppose $(\vec{b},\vec{a})$ and $(\pvec{b}',\pvec{a}')$ are stable pairs with $\pvec{a}'\leq \vec{a}$ and $\vec{b} \leq \pvec{b}'$. Then $(\vec{b},\vec{a})$ and $(\pvec{b}',\pvec{a}')$ are connected.
\end{lemma}

\begin{proof}
We check that $(\vec{b},\pvec{a}')$ is a realizable pair using Theorem \ref{thm:main_realizability}. Let $A_{\vec{u}}$, $B_{\vec{v}}$, $S_{\vec{v},\vec{u}}$, denote $A,B,S$ as usual evaluated on the kernel-quotient pair $\vec{v},\vec{u}$. Then $S_{\vec{b},\pvec{a}'}(\mu, \nu)\geq S_{\vec{b},\vec{a}}(\mu, \nu), S_{\pvec{b}',\pvec{a}'}(\mu, \nu)$, so realizability of $(\pvec{b}',\pvec{a}')$ immediately implies $A_{\pvec{a}'}(\mu, \nu)\geq 0$ where necessary, and similarly, realizability of $(\vec{b},\vec{a})$ immediately implies $B_{\vec{b}}(\mu, \nu)\geq 0$ where necessary. The conclusion now follows, since $\Sigma_{\vec{b}}(\cS) \cap \Sigma_{\pvec{a}'}(\cQ)$ is in the closures of $\Sigma_{\vec{b}}(\cS) \cap \Sigma_{\vec{a}}(\cQ)$ and $\Sigma_{\pvec{b}'}(\cS) \cap \Sigma_{\pvec{a}'}(\cQ)$. 
\end{proof}

The other main tool will be \textit{iterative balancing}.

\begin{definition}
Suppose $(\vec{b},\vec{a})$ is a realizable pair. We define the \textit{iterative balancing of $(\vec{b},\vec{a})$} to be the pair $\IB(\vec{b},\vec{a})$ obtained by the following process:
\begin{itemize}
    \item define $\pvec{b}^{(0)}:=\vec{b}$ and $\pvec{a}^{(0)}:=\vec{a}$;
    \item for $i\geq1$, define $\vec{a}^{(i)}$ to be the most balanced pair $\pvec{a}'$ such that $(\vec{b}^{(i-1)},\pvec{a}')$ is realizable;
    \item for $i\geq1$, define $\vec{b}^{(i)}$ to be the most balanced pair $\pvec{b}'$ such that $(\pvec{b}',\vec{a}^{(i)})$ is realizable;
    \item define \[\IB(\vec{b},\vec{a}):=(\pvec{b}^{(k)},\pvec{a}^{(k)}),\] where $k$ is chosen sufficiently large so that $(\pvec{b}^{(k)},\pvec{a}^{(k)})=(\pvec{b}^{(k-1)},\pvec{a}^{(k-1)})$.
\end{itemize}
Such a $k$ will exist since otherwise there would be no most balanced tuple of a given degree.
\end{definition}

\begin{lemma}\label{lem:connecting_IB}
Let $(\pvec{b}',\pvec{a}')$ is a stable pair, and suppose $(\vec{b},\vec{a})$ is a realizable pair connected to $(\pvec{b}',\pvec{a}')$ with $\vec{a}>\pvec{a}'$. Then $(\pvec{b}'',\pvec{a}''):=\IB(\vec{b},\vec{a})$ is a stable pair connected to $(\pvec{b}',\pvec{a}')$ with $\pvec{a}''>\pvec{a}'$.
\end{lemma}

\begin{proof}
It will suffice to show that $\IB(\vec{b},\vec{a})$ and $(\vec{b},\vec{a})$ are connected. Let $k$ be the index such that $\IB(\vec{b},\vec{a})=(\pvec{b}^{(k)},\pvec{a}^{(k)})$ as in the process described in the definition above. Note that for $i\geq 1$, $(\pvec{b}^{(i-1)},\pvec{a}^{(i-1)})$ is in the closure of $(\pvec{b}^{(i-1)},\pvec{a}^{(i)})$ which, in turn, is in closure of $(\pvec{b}^{(i-1)},\pvec{a}^{(i-1)})$. Then by induction, $(\vec{b},\vec{a})=(\pvec{b}^{(0)},\pvec{a}^{(0)})$ is in the closure of $(\pvec{b}^{(k)},\pvec{a}^{(k)}) =\IB(\vec{b},\vec{a})$, and we are done.
\end{proof}

Our strategy for showing connectedness will be to show that any stable pair is connected to the strongly stable pair $(\pvec{b}^0,\pvec{a}^0)$ with the most balanced possible quotient splitting type. We give a means of using Lemmas \ref{lem:connecting_pairs} and \ref{lem:connecting_IB} to produce from any pair not equal to $(\pvec{b}^0,\pvec{a}^0)$ another pair with a more balanced quotient splitting type. Before detailing this method in full, we demonstrate it with an example:

\begin{example}
    Let $\vec{e}=(0,4,10,13,15,20)$. Then the pair $\vec{b}=(-1,9,14)$, $\vec{a}=(5,14,21)$ is stable, coming from the partition $\Lambda= ((1),(2),\ (3),(4),\ (5),(6))$ with $\vec{\delta}=(1,1,1)$. Taking the same partition with $\vec{\delta}=(2,1,0)$ gives the stable pair $\vec{b}=(-2,9,15)$, $\vec{a}=(6,14,20)$ and $\vec{\delta}=(3,0,0)$ gives the stable pair $\vec{b}=(-3,10,15)$, $\vec{a}=(7,13,20)$

    Now we have a stable pair coming from a partition with 3 parts where $\vec{\delta}$ is zero in entries indexed greater than 1. Lemma \ref{lem:to_one_block} shows how to connect this to a realizable pair with more balanced quotient splitting type, which in this case, is $\vec{b}=(-8,10,20)$, $\vec{a}=(8,13,19)$. Iterative balancing connects this to the pairs:
    \begin{center}
    \begin{tabular}{ c|c } 
$\vec{b}$ & $\vec{a}$ \\
\hline
$(-4,6,20)$ & $(8,13,19)$ \\
$(-4,6,20)$ & $(8,16,16)$ \\
\end{tabular}
\end{center}
The latter row is stable, coming from the partition $((1),(2),\ (3),(4,5))$ with $m'=1$ and $\delta=(4,4)$, so we now return to the process given in Lemma \ref{lem:killing_higher_delta} (which involves more iterative balancing) to obtain the sequence of connected pairs:
    \begin{center}
    \begin{tabular}{ c|c } 
$\vec{b}$ & $\vec{a}$ \\
\hline
$(-4,6,20)$ & $(8,16, 16)$ \\
$(-5,7,20)$ & $(9,15,16)$ \\
$(-5,8,20)$ & $(10,15,15)$ \\
$(0,2,20)$ & $(10,15,15)$ \\
$(0,2,20)$ & $(12,13,15)$ \\
\end{tabular}
\end{center}
The last pair obtained is stable, and comes from the partition $((1,2),(3,4,5))$ with $m'=1$. By Lemma \ref{lem:connecting_pairs}, this pair is connected to the stable pair associated to the partition $((1),(2,3,4))$ with $m'=2$, namely $\pvec{b}^0=(-13,15,20)$, $\pvec{a}^0=(13,13,14)$, which has balanced quotient splitting type, as desired.
\end{example}

\begin{lemma} \label{lem:killing_higher_delta}
Suppose $(\vec{b},\vec{a})$ is a stable pair corresponding to $m', n',\Lambda,\vec{\delta}$ with $\delta_i>0$ for some $i>1$. Then there exists a stable pair $(\pvec{b}',\pvec{a}')$ connected to $(\vec{b},\vec{a})$ with $\pvec{a}'>\vec{a}$.
\end{lemma}

\begin{proof}
Choose $i>1$ such that $\delta_i>0$. Then \[(\pvec{b}'',\pvec{a}''):=\text{Pair}(m', n', \Lambda,(\delta_1+1,\dots,\delta_i-1,\dots,\delta_r))\]
is realizable, and connected to $(\vec{b},\vec{a})$ by Lemma \ref{lem:connecting_pairs}, with $\pvec{a}''>\vec{a}$ (note that the entries of $\pvec{a}''$ are ordered, since $\alpha(Q_1, \delta_1)_{\text{last}}<e_{P_{2,1}}<\alpha(Q_i,\delta_i)_1$ and entries of $\pvec{a}''$ will differ from those of $\vec{a}$ will differ by at most 1). Now we can define $(\pvec{b}',\pvec{a}'):=\IB(\pvec{b}'',\pvec{a}'')$, and Lemma \ref{lem:connecting_IB} finishes. 
\end{proof}

Repeated application of Lemma \ref{lem:killing_higher_delta} gives:

\begin{corollary} \label{cor:killing_higher_delta}
Every stable pair $(\vec{b},\vec{a})$ is connected to a stable pair $\pvec{b}',{a'}$ corresponding to $m', n',\Lambda,\vec{\delta}$ with $\delta_i=0$ for all $i>1$ with $\pvec{a}'\geq \vec{a}$.
\end{corollary}

\begin{lemma} \label{lem:to_one_block}
Suppose $(\vec{b},\vec{a})=\text{Pair}(m', n',\Lambda,\vec{\delta})$ with $\Lambda=(P_1,Q_1,\dots,P_r,Q_r)$, $r\geq 2$, and $\delta_i=0$ for all $i>1$. Then there exists a stable pair $(\pvec{b}',\pvec{a}')$ connected to $(\vec{b},\vec{a})$ with $\pvec{a}'>\vec{a}$.
\end{lemma}

\begin{proof}
Let $d=e_{Q_{r,1}}-e_{Q_{r,1}-1}$. Let $\pvec{a}'$ be obtained from $\vec{a}$ by replacing the entry $a_{j_0}$ equal to $e_{Q_{r,1}}$ with $e_{Q_{r,1}}-1$ and $\alpha(Q_1, \delta_1)$ with $\alpha(Q_1, \delta_1+1)$ (here, $r>2$ implies $e_{Q_{r,1}}>e_{P_{2,1}}$ implies the entries of $\pvec{a}'$ are ordered as in Lemma \ref{lem:killing_higher_delta}) , and $\pvec{b}'$ obtained from $\vec{b}$ by replacing $e_{Q_{r,1}-1}$ with $e_{Q_{r,1}}$ and $b_1$ with $b_1-d$.

We can obtain a balancing datum for $(\pvec{b}',\pvec{a}')$ by starting with a balancing datum for $(\vec{b},\vec{a})$, swapping $\gs(\gs^{-1}(Q_{r,1}))$ and $\tau(\tau^{-1}(Q_{r,1}-1))$, increasing $\G(1,j)$ by 1 where $a_j$ is the entry of $\alpha(Q_1, \delta_1)$ that was changed going from $\vec{a}$ to $\pvec{a}'$, and increasing $\G(1,j_0)$ by $d-1$. Then $(\pvec{b}',\pvec{a}')$ is realizable and connected to $(\vec{b},\vec{a})$ by Lemma \ref{lem:connecting_pairs}, and $\pvec{a}'>\vec{a}$. Now the conclusion follows from applying Lemma \ref{lem:connecting_IB}.
\end{proof}

\textit{Proof of Theorem \ref{thm:intro_connected}} 
We show that any stable pair $(\vec{b},\vec{a})$ is connected to the strongly stable pair $(\pvec{b}^0,\pvec{a}^0)$ with the most balanced possible quotient splitting type. 

Starting with $(\vec{b},\vec{a})$, applying Corollary \ref{cor:killing_higher_delta} followed by Lemma \ref{lem:to_one_block} and choosing a strongly stable pair containing the result in its closure, we can obtain a strongly stable pair $(\pvec{b}'',\pvec{a}'')$ with $\pvec{a}''>\vec{a}$, unless $(\vec{b},\vec{a})$ is a strongly stable pair of the form $\text{Pair}(m', n',(P_1,Q_1),(\delta_1))$. Henceforth assume without loss of generality that $(\vec{b},\vec{a})$ is strongly stable and of this form. The pair $(\pvec{b}^0,\pvec{a}^0)$ is of the same form, so let it be given by $\text{Pair}({m^0}', {n^0}',(P^0_1,Q^0_1),(\delta^0_1))$. 

Now, by assumption $\pvec{a}^0\geq\vec{a},$ so $n'\geq {n^0}'$. If ${n'}= {n^0}'$, then $\pvec{a}^0=\vec{a},$ so we are done. Now suppose $n'> {n^0}'$. If $m'\geq {m^0}'$, then $\vec{b}\leq \pvec{b}^0$, but this is is a contradiction since it implies $(\vec{b},\vec{a})$ is not strongly stable. On the other hand, if $m'< {m^0}'$, then $\vec{b}> \pvec{b}^0$, in which case Lemma \ref{lem:connecting_pairs} shows that $(\vec{b},\vec{a})$ and $(\pvec{b}^0,\pvec{a}^0)$ are connected. This completes the proof.
\qed

\newpage

\printbibliography

\end{document}